\documentclass[a4paper]{article}
\usepackage{amssymb,amsmath,enumerate,graphicx,tikz}
\usepackage{amsthm}
\usepackage[utf8]{inputenc}

\makeatletter
\newtheorem*{rep@theorem}{\rep@title}
\newcommand{\newreptheorem}[2]{%
\newenvironment{rep#1}[1]{%
 \def\rep@title{#2 \ref{##1}}%
 \begin{rep@theorem}}%
 {\end{rep@theorem}}}
\makeatother

\newreptheorem{theorem}{Theorem}
\newreptheorem{lemma}{Lemma}

\newtheorem{definition}{Definition}

\newtheorem{proposition}[definition]{Proposition}
\newtheorem{theorem}[definition]{Theorem}
\newtheorem{corollary}[definition]{Corollary}
\newtheorem{lemma}[definition]{Lemma}
\newtheorem{conjecture}[definition]{Conjecture}

\newcommand{\comment}[1]{}
\newcommand{\N}{\mathbb N}

\newcommand{\Z}{\mathbb Z}
\newcommand{\es}{\emptyset}
\newcommand{\EP}{Erd\H{o}s-P\'osa}
\newcommand{\sparse}{robust}

\newcommand{\cT}{\mathcal{T}}
\newcommand{\cP}{\mathcal{P}}
\newcommand{\cQ}{\mathcal{Q}}

\newcommand{\cC}{\mathcal{C}}
\newcommand{\cG}{\mathcal{G}}
\newcommand{\G}{\Gamma}
\newcommand{\g}{\gamma}
\newcommand{\zT}{\mathcal{T}}

\newcommand{\dirG}{\overrightarrow{G}}
\newcommand{\dirW}{\overrightarrow{W}}
\newcommand{\dirWone}{\overrightarrow{W_1}}
\newcommand{\dirH}{\overrightarrow{H}}
\newcommand{\dirU}{\overrightarrow{U}}

\DeclareMathOperator{\head}{head}
\DeclareMathOperator{\tail}{tail}

\tikzset{
    declare function={Floor(\x)=round(\x-0.5);}
}

\newcommand{\sm}{\setminus}

\title{A unified Erd\H{o}s-P\'osa theorem for constrained cycles}
\author{Tony Huynh, Felix Joos and Paul Wollan
\thanks{This research is supported by the European Research Council under the European Unions Seventh Framework Programme (FP7/2007-2013)/ERC Grant Agreement no. 279558.}
}
\date{}

\begin{document}
\maketitle

\begin{abstract}
A \emph{$(\G_1,\G_2)$-labeled graph} is an oriented graph with its edges labeled by elements of the direct sum of two  
groups $\G_1,\G_2$.  
A cycle in such a labeled graph is \emph{$(\G_1,\G_2)$-non-zero} if it is non-zero in both coordinates.  
Our main result is a generalization of the Flat Wall Theorem of Robertson and Seymour to $(\G_1,\G_2)$-labeled graphs. 
As an application, we determine all canonical obstructions to the Erd\H{o}s-P\'osa property for $(\G_1,\G_2)$-non-zero cycles in $(\G_1,\G_2)$-labeled graphs. The obstructions imply that the half-integral Erd\H{o}s-P\'osa property always holds for $(\G_1,\G_2)$-non-zero cycles.   

Moreover, our approach gives a unified framework for proving packing results for constrained cycles in graphs.  
For example, as immediate corollaries we recover the \EP{} property for cycles and $S$-cycles and the half-integral \EP{} property for odd cycles and odd $S$-cycles.  Furthermore, we recover Reed's Escher-wall Theorem.

We also prove many new packing results as immediate corollaries.  
For example, we show that the half-integral \EP{} property holds for cycles not homologous to zero, odd cycles not homologous to zero, and $S$-cycles not homologous to zero.  Moreover, the (full) \EP{} property holds for $S_1$-$S_2$-cycles and cycles not homologous to zero on an orientable surface.  Finally, we also describe the canonical obstructions to the \EP{} property for cycles not homologous to zero and for odd $S$-cycles.  
\end{abstract}

\section{Introduction}

Erd\H{o}s and P\'osa proved in a seminal paper~\cite{EP62} that for every positive integer $k$, 
there exists a constant $f(k)$ such that every graph $G$ 
either contains $k$ pairwise disjoint\footnote{In this paper disjoint always means vertex-disjoint.} cycles, or a set $X \subseteq V(G)$ of size at most $f(k)$ such that $G-X$ has no cycle.  This result has had numerous extensions and generalizations to cycles satisfying further constraints.  Specific examples of families of cycles which have been studied are: cycles of odd length \cite{Ree99},  cycles of length at least $\ell$ for some $\ell\in \N$ \cite{BBR07,FH13}, disjoint $S$-cycles where each cycle intersects a prescribed set of vertices $S$ \cite{KKM11,PW12},  $S$-cycles of odd length \cite{KK13}, and $S$-cycles of length at least $\ell$ again for some fixed value $\ell$ \cite{BJS14}.  In each case, the goal is to show the \emph{Erd\H{o}s-P\'osa property} for a given family of cycles: that is, the existence of a function $f$ such that for all $k$, every graph either has $k$ disjoint cycles satisfying the desired constraints, or a set $X$ of at most $f(k)$ vertices intersecting every such cycle.

In this paper, we consider the analogous problem in group-labeled graphs.  
This approach provides a unified framework for proving \EP{} results for families of cycles with additional constraints on the cycles. 
In particular, we obtain all of the above results as corollaries of our main theorems.  

In this article, the graphs we consider may have loops and parallel edges.  An \emph{oriented graph} is a graph $\overrightarrow{G}$ along with two functions \emph{$\head_{\overrightarrow{G}}$} and \emph{$\tail_{\overrightarrow{G}}$} from $E(G)$ to $V(G)$ such that for every non-loop edge $e$ with distinct endpoints $u$ and $v$, we have $\head_{\overrightarrow{G}}(e) \in \{u, v\}$, and $\tail_{\overrightarrow{G}}(e) =\{u,v\}\setminus \head_{\overrightarrow{G}}(e)$.  
In the case when $e$ is a loop on a vertex $v$, we have $\{\head_{\overrightarrow{G}}(e), \tail_{\overrightarrow{G}}(e)\}  = \{v\}$.  Given an oriented graph $\overrightarrow{G}$, the undirected graph $G$ with the same vertex and edge set as $\overrightarrow{G}$ is the graph obtained by \emph{forgetting} the orientations.  Notationally, we will always indicate oriented graphs with an arrow and undirected graphs without.  When there can be no confusion, we will simply use $tail$ and $head$ in the place of $\tail_{\overrightarrow{G}}$ and $\head_{\overrightarrow{G}}$.

Let $\Gamma, \Gamma_1, \Gamma_2$ be (possibly infinite, possibly non-abelian) groups. Even though our theorems hold for non-abelian groups, we will always use $+$ to denote the group operation and $0$ for the identity element.  This follows the convention in \cite{CGGGLS06}, and is mainly because the adjective \emph{non-zero} is so ubiquitous in the literature.  
A \emph{$\Gamma$-labeled} graph is a pair $(\overrightarrow{G}, \gamma)$, where $\overrightarrow{G}$ is an oriented graph and $\gamma:E(\overrightarrow{G}) \to \Gamma$. Group-labeled graphs have been extensively studied; see~\cite{CGGGLS06, Huynh09, Zaslavsky89, Zaslavsky91}.   
Let $v$ be an end of an edge $e$.  We define $\gamma(e, v):=\gamma(e)$ if $v=\head(e)$ and 
$\gamma(e,v)=-\gamma(e)$ if $v=\tail(e)$.   If $\overrightarrow{H}$ is a subgraph of $\dirG$, we abuse notation by letting $\g$ also denote the restriction of $\g$ to $E(\overrightarrow{H})$.

A \emph{walk} in $(\overrightarrow{G}, \gamma)$ is a walk in $G$. Let $W=v_0 e_1 v_1 e_2 v_2 \dots e_\ell v_\ell$ be a walk in $(\overrightarrow{G}, \gamma)$. The \emph{length} of $W$ is $\ell$, and its \emph{ends} are $v_0$ and $v_\ell$.  
We say that $W$ is a \emph{path} if $v_0, \dots, v_\ell$ are distinct and is a \emph{cycle} if $v_0=v_\ell$ and $v_1, \dots, v_\ell$ are distinct. The \emph{group-value}
of $W$, denoted $\gamma(W)$, is defined to be $\gamma(e_1, v_1)+ \dots + \gamma(e_\ell, v_\ell)$. 
We say $W$ is \emph{$\G$-zero} if $\gamma(W)=0$. Otherwise, we say that $W$ is \emph{$\G$-non-zero}. Note that if $W$ is a cycle, it is naturally equipped with a starting point and an orientation (since it is a walk).  
We remark that if $C_1$ and $C_2$ are cycles such that $E(C_1)=E(C_2)$, then $C_1$ is $\G$-zero if and only if $C_2$ is $\G$-zero (see Section~\ref{sec:background} for a proof of this claim).
Thus, whether a cycle $C$ is $\G$-zero or $\G$-non-zero only depends on its edge set.  
This is easy to see if $\G$ is abelian, since in this case $\g(C_1) \in \{\g(C_2), -\g(C_2)\}$.

We consider a stronger notion of non-zero cycles in the case that $(\overrightarrow{G}, \g)$ is a $(\G_1 \oplus \G_2)$-labeled graph\footnote{We denote by $\G_1 \oplus \G_2$ the direct sum of $\G_1$ and $\G_2$. The elements of $\G_1 \oplus \G_2$ can be written as $(\alpha_1,\alpha_2)$ with $\alpha_i\in \G_i$ and the group operation is induced by the coordinate-wise group operation of $\G_1$ and $\G_2$, respectively.}. 
For $i \in [2]$, we let $\g_i$ be the projection of $\g$ onto $\G_i$.
A walk $W$ is \emph{$\G_i$-zero} if $\gamma_{i}(W)=0$.  Otherwise, $W$ is \emph{$\G_i$-non-zero}.
If $W$ is $\G_1$-non-zero and $\G_2$-non-zero,
then we say $W$ is \emph{$(\G_1,\G_2)$-non-zero}.

Let $\mathcal{G}$ be a class of $(\G_1 \oplus \G_2)$-labeled graphs.  We say that $\mathcal{G}$ has the \emph{\EP-property for $(\G_1,\G_2)$-non-zero cycles} if there exists a function $f_{\cal{G}}(k)$ such that every $(\overrightarrow{G}, \g) \in \mathcal{G}$ either contains $k$ pairwise disjoint $(\G_1,\G_2)$-non-zero cycles, or a set $X \subseteq V(G)$ of size at most $f_{\mathcal{G}}(k)$ such that $(\overrightarrow{G}-X, \g)$ has no $(\G_1,\G_2)$-non-zero cycle. We say that $f_{\mathcal{G}}(k)$ is an \emph{\EP\ function for $(\G_1,\G_2)$-non-zero cycles in $\cG$}.

In Section \ref{sec:obstructions}, we show that the \EP-property for $(\G_1,\G_2)$-non-zero cycles does \emph{not} hold for the class of all $(\G_1 \oplus \G_2)$-labeled graphs. However, the following weakening of the \EP-property does hold.   We say that a class $\mathcal{G}$ of $(\G_1 \oplus \G_2)$-labeled graphs has the \emph{half-integral \EP-property for $(\G_1,\G_2)$-non-zero cycles} if there exists a function $f_{\mathcal{G}}(k)$ such that for every $(\overrightarrow{G}, \g) \in \mathcal{G}$ either
\begin{enumerate}[(i)]
\item there is a collection $\cC$ of $k$ $(\G_1,\G_2)$-non-zero cycles such that each vertex of $(\dirG, \g)$ is contained in at most two members of $\cC$, or
\item  there is a set $X \subseteq V(G)$ of size at most $f_{\mathcal{G}}(k)$ such that $(\overrightarrow{G}-X, \g)$ has no $(\G_1,\G_2)$-non-zero cycle. 
\end{enumerate}
We say that $f_{\mathcal{G}}(k)$ is a \emph{half-integral \EP\ function for $(\G_1,\G_2)$-non-zero cycles in $\cG$}.

\begin{theorem}\label{thm: EP hi cycles}
For every integer $k$, there exists an integer $f(k)$ with the following property.  
Let $\G_1$ and $\G_2$ be  groups
and let $(\dirG, \g)$ be a $(\G_1 \oplus \G_2)$-labeled graph.  
Then, $(\dirG, \g)$ contains $k$ $(\G_1,\G_2)$-non-zero cycles such that each vertex of $(\dirG, \g)$ is in at most two of these cycles, 
or there exists a set $X$ of at most $f(k)$ vertices of $G$ such that $(\dirG-X, \g)$ does not contain any $(\G_1,\G_2)$-non-zero cycle.  
\end{theorem}

We also obtain the (full) \EP-property for $(\G_1,\G_2)$-non-zero cycles for the following restricted class of $(\G_1 \oplus \G_2)$-labeled graphs. 
For $i\in[2]$, let $Z_i$ be the set of edges of a $(\G_1 \oplus \G_2)$-labeled graph $\dirG$ that lie in a $\G_i$-zero cycle.
We say $(\dirG, \g)$ is \emph{\sparse} if the following holds for all $i \in [2]$: 
for all $\G_i$-non-zero cycles $C_1$ and $C_2$ such that 
\begin{itemize}
	\item $E(C_1) \neq E(C_2)$
	\item $\emptyset\neq E(C_1)\cap E(C_2)\subseteq Z_i$, and
	\item $C_1$ and $C_2$ have the same starting vertex,
\end{itemize}
we have $\g_i(C_1) \neq \g_i(C_2)$.
Note that this condition is easier to check if $(\G_1 \oplus \G_2)$ is abelian
because $\g(C_i)$ is independent of the starting vertex.

\begin{theorem}\label{thm: EP groups sparse}
For every integer $k$, there exists an integer $f(k)$ with the following property.  
Let $\G_1$ and $\G_2$ be  groups 
and let $(\dirG, \g)$ be a \sparse\ $(\G_1 \oplus \G_2)$-labeled graph.
Then, $(\dirG, \g)$ contains $k$ disjoint $(\G_1,\G_2)$-non-zero cycles or there exists a set of at most $f(k)$ vertices of $(\dirG, \g)$
such that $(\dirG-X, \g)$ does not contain any $(\G_1,\G_2)$-non-zero cycle.  
\end{theorem}

Note that in Theorem \ref{thm: EP hi cycles} and \ref{thm: EP groups sparse}, the \EP\ functions do not depend on the groups $\G_1$ and $\G_2$. 
Also, if either $\G_1$ or $\G_2$ is the trivial group, then there are no $(\G_1,\G_2)$-non-zero cycles, so the \EP{} property holds trivially.
In addition, if we set $\G_1=\G_2$ and all edges have the same label in both coordinates,
then $(\G_1,\G_2)$-non-zero cycles coincide with $\G$-non-zero cycles.

We prove Theorem \ref{thm: EP hi cycles} and \ref{thm: EP groups sparse} via a structure theorem for $(\G_1 \oplus \G_2)$-labeled graphs, which
is a refinement of the Flat Wall Theorem of Robertson and Seymour \cite[Theorem 9.8]{RS95}.  Our structure theorem (Theorem \ref{thm: main wall}) is likely of independent interest.
For example, Theorem \ref{thm: main wall} provides canonical obstructions to the \EP{} property for general constrained cycles that are analogous to the \emph{Escher-walls} described by Reed \cite{Ree99}.  These obstructions are defined in Section \ref{sec:obstructions}. 

Our proof of Theorem \ref{thm: main wall} extends the techniques developed in \cite{GG09} and \cite{GGRSV09}.  We introduce a new notion of \emph{$\G$-odd clique minors} for $\G$-labeled graphs, which is useful for 
attacking problems for non-zero cycles. Our definition 
agrees with the usual notion of odd clique minors when $\G=\Z / 2 \Z$, but is weaker than the notion of a $\G$-labeled clique used in 
\cite{GG09}. Our notion is well-defined even in the case that $\G$ is infinite (note that it will be necessary to consider infinite $\G$ for some of our applications).  

The original draft of Theorem 1 and 2 stated the results under the assumption that the group $\Gamma$ is abelian; 
an early manuscript of \cite{LRS17} led us to the observation that both theorems (and their respective proofs) hold without the assumption that $\Gamma$ is abelian.

We suspect that Theorem \ref{thm: main wall} will have further applications outside those discussed in this paper.  However, as it is rather technical, we defer the statement of Theorem \ref{thm: main wall} until Section \ref{sec:flatwall} and the proof until Section \ref{sec:thmmain}. 

We instead discuss some applications of Theorem \ref{thm: EP hi cycles} and \ref{thm: EP groups sparse} in the next section. 
The rest of the paper is organized as follows.  
In Section \ref{sec:background}, we discuss some preliminaries.  
In Section \ref{sec:obstructions}, we provide a canonical set of obstructions to the \EP{} 
for $(\G_1,\G_2)$-non-zero cycles. 
In Section \ref{sec:oddkt}, 
we introduce our notion of $\G$-odd clique minors, and prove a structure theorem for $\G$-labeled graphs without a $\G$-odd clique minor.  
In Section \ref{sec:flatwall}, we state our version of the Flat Wall Theorem for $(\G_1,\G_2)$-group-labeled graphs (Theorem \ref{thm: main wall}).  
In Section \ref{sec:cleaning} and \ref{sec:walllemmas}, we prove some lemmas to be used in the proof of Theorem \ref{thm: main wall}.  
We prove Theorem~\ref{thm: main wall} in Section \ref{sec:thmmain}. 
We then easily derive Theorem \ref{thm: EP hi cycles} and \ref{thm: EP groups sparse} from Theorem \ref{thm: main wall} in Section \ref{sec:deriving} and finish with some further applications.

\section{Applications}\label{sec:applications}

In this section we illustrate the numerous corollaries of Theorem~\ref{thm: EP hi cycles} and \ref{thm: EP groups sparse}.

\textbf{Cycles.} Let $G$ be a graph and let $\dirG$ be an arbitrary orientation of $G$.  Let $e_1, \dots, e_m$ be an enumeration of $E(\dirG)$ and define $\g: E(\dirG) \to \Z \oplus \Z$ by $g(e_i)=(2^i, 2^i)$.  Note that $(\dirG, \g)$ does not contain any $\G_i$-zero cycles, and therefore $(\dirG, \g)$ is clearly \sparse{}.   Since every cycle in $G$ corresponds to a $(\G_1,\G_2)$-non-zero cycle in $(\dirG, \g)$, Theorem \ref{thm: EP groups sparse} implies the original theorem of Erd\H{o}s and P\'osa.  

\textbf{Odd Cycles.} Let $G$ be a graph.  An \emph{odd cycle} of $G$ is a cycle with an odd number of vertices.  Thomassen \cite{T88} proved that the \EP-property does not hold for odd cycles.  On the other hand, Reed \cite{Ree99} showed that the set of odd cycles has the half-integral \EP\ property.
\begin{theorem}[Reed \cite{Ree99}]\label{thm: reed}
There exists a function $f$ such that for every $k \ge 1$ and graph $G$, either $G$ contains $k$ odd cycles $C_1, \dots, C_k$ such that every vertex is contained in at most two distinct $C_i$, or alternatively, there exists a set $X \subseteq V(G)$ such that $G-X$ is bipartite and $|X|\leq f(k)$.
\end{theorem}
We derive Theorem \ref{thm: reed} as a corollary of Theorem \ref{thm: EP hi cycles} as follows.  Let 
$\dirG$ be an arbitrary orientation of $G$.  
Define $\g: E(\dirG) \to \Z / 2\Z \oplus \Z / 2\Z$ by $\g(e)=(1,1)$ for all $e\in E(\dirG)$.  
Finish by observing that every odd cycle of $G$ corresponds to a $(\G_1,\G_2)$-non-zero cycle in $(\dirG, \g)$.   

\textbf{$S$-cycles.} Given a graph $G$ and a fixed subset $S \subseteq V(G)$, an \emph{$S$-cycle} is a cycle in $G$ containing at least one vertex of $S$.  Recent work of Kakimura, Kawarabayashi and Marx \cite{KKM11} and Pontecorvi and Wollan \cite{PW12} shows that the \EP\ property holds for the family of $S$-cycles.  
\begin{theorem}[\cite{KKM11, PW12}] \label{thm: Scycles}
There exists a function $f:\N\to\N$ such that for every graph $G$, every $S \subseteq V(G)$, and every positive integer $k$, either $G$ has $k$ disjoint $S$-cycles or there exists a set of at most $f(k)$ vertices intersecting every $S$-cycle in $G$.
\end{theorem}

We derive Theorem \ref{thm: Scycles} as a corollary to Theorem \ref{thm: EP groups sparse} as follows.  Let $G$ be a graph and $S \subseteq V(G)$.  Let $\overrightarrow{G}$ be an arbitrary orientation of $G$ and let $e_1, \dots, e_m$ be an enumeration of $E(\dirG)$. Define $\g: E(\overrightarrow{G}) \to \mathbb{Z} \oplus \mathbb{Z}$ by $\g(e_i)=(2^i, 2^i)$, if $e_i$ has at least one end in $S$, and $\g(e_i)=(0,0)$, otherwise.  Let $Z_i$ be the set of edges of $(\dirG, \g)$ that are contained in a $\G_i$-zero cycle and suppose $C_1$ and $C_2$ are distinct $\G_i$-non-zero cycles with $E(C_1) \cap E(C_2) \subseteq Z_i$.  Let $C_1'$ and $C_2'$ be the corresponding cycles in $G$.  Note that $C_1'$ and $C_2'$ both meet $S$, but they do not share any edges which have an end in $S$.  Therefore, $\g_i(C_1) \notin \{\g_i(C_2), -\g_i(C_2)\}$, and so $(\dirG, \g)$ is \sparse{}.  Clearly, every $S$-cycle in $G$ corresponds to a $(\G_1,\G_2)$-non-zero cycle in $(\dirG, \g)$, which proves Theorem \ref{thm: Scycles}.  

\textbf{Cycles not homologous to zero.}  Let $G$ be a graph embedded in a (possibly non-orientable) surface $\Sigma$ and let $\mathcal{H}(\Sigma)$ be the (first) homology group of $\Sigma$.  A cycle $C$ in $G$ is \emph{not homologous to zero} if it is non-zero in $\mathcal{H}(\Sigma)$.  We use the following easy proposition. 
We give a proof at the end of Section~\ref{sec:background}.
\begin{proposition} \label{homologylabel}
Let $G$ be a graph embedded on a  surface $\Sigma$ and let $\G$ be the homology group of $\Sigma$. Then there exists an orientation $\dirG$ of $G$ and a labeling $\g:E(\dirG) \to \Gamma$ such that the group-value of every closed walk $W$ in $(\dirG, \g)$ is precisely the homology class of $W$.   
\end{proposition}

  By the previous proposition, there is an orientation $\dirG$ of $G$ and a labeling $\g:E(\dirG) \to \mathcal{H}(\Sigma) \oplus \mathcal{H}(\Sigma)$, such that a cycle $C$ in $G$ is not homologous to zero if and only if the corresponding cycle in $(\dirG, \g)$ is $(\G_1,\G_2)$-non-zero.  Therefore, as a corollary to Theorem \ref{thm: EP hi cycles}, we obtain the half-integral \EP{} property for cycles not homologous to zero, which we believe is new.

\begin{theorem} \label{thm: homology}
For every integer $k$, there exists an integer $f(k)$ with the following property.  Every graph $G$ embedded in a surface $\Sigma$ either contains $k$ cycles $C_1, \dots, C_k$ such that each $C_i$ is not homologous to zero and each vertex of $G$ is contained in at most two distinct $C_i$, or there is a set of at most $f(k)$ vertices of $G$ such that every cycle of $G-X$ is homologous to zero.  
\end{theorem}

\textbf{Two constraints.} 
The real power of Theorem \ref{thm: EP hi cycles} and Theorem \ref{thm: EP groups sparse} is that we can prove \EP{} results for cycles satisfying \emph{two}
constraints. We now give some examples of this phenomenon.  The first is a recent theorem of Kawarabayashi and Kakimura~\cite{KK13}.

\begin{theorem}[Odd $S$-cycles, \cite{KK13}] \label{thm: odd S}
There exists a function $f:\N\to\N$ such that for every graph $G$, every set $S \subseteq V(G)$, and $k \ge 1$, either there exist $k$ odd $S$-cycles such that every vertex is in at most two of them, or there exists a set of at most $f(k)$ vertices intersecting every odd $S$-cycle.
\end{theorem}

We derive Theorem \ref{thm: odd S} as a corollary to Theorem \ref{thm: EP hi cycles} as follows. Let $G$ be a graph and $S \subseteq V(G)$.  Let $\overrightarrow{G}$ be an arbitrary orientation of $G$ and let $e_1, \dots, e_m$ be an enumeration of $E(\dirG)$.  Define $\g: E(\overrightarrow{G}) \to \Z / 2\Z \oplus \Z$ by $\g(e_i)=(1, 2^i)$, if $e_i$ has at least one end in $S$, and $\g(e_i)=(1,0)$, otherwise.  Theorem \ref{thm: odd S} follows by observing that every odd $S$-cycle in $G$ corresponds to a $(\G_1,\G_2)$-non-zero cycle in $(\dirG, \g)$.

Similarly, we also obtain the half-integral \EP{} property for odd cycles not homologous to zero and for $S$-cycles not homologous to zero.

\begin{theorem}[Odd cycles not homologous to zero] \label{thm: oddhomology}
For every integer $k$, there exists an integer $f(k)$ with the following property.  Every graph $G$ embedded in a surface $\Sigma$ either contains $k$ odd cycles $C_1, \dots, C_k$ such that each $C_i$ is not homologous to zero and each vertex of $G$ is contained in at most two distinct $C_i$, or there is a set of at most $f(k)$ vertices of $G$ meeting all odd cycles in $G$ that are not homologous to zero.
\end{theorem}

\begin{theorem}[$S$-cycles not homologous to zero] \label{thm: Shomology}
For every integer $k$, there exists an integer $f(k)$ with the following property.  For every graph $G$ embedded in a surface $\Sigma$ and for all $S \subseteq V(G)$, either $G$ contains $k$ cycles $C_1, \dots, C_k$ such that each $C_i$ is an $S$-cycle not homologous to zero and each vertex of $G$ is contained in at most two distinct $C_i$, or there is a set of at most $f(k)$ vertices of $G$ meeting all $S$-cycles in $G$ that are not homologous to zero.
\end{theorem}

\textbf{Remark.} The functions in Theorem \ref{thm: homology}, \ref{thm: oddhomology}, and \ref{thm: Shomology} do not depend on the surface $\Sigma$.

We finish by giving a new example of a $(\G_1,\G_2)$-constrained cycle problem where we obtain the (full) \EP{} property.  
Let $G$ be a graph and $S_1$ and $S_2$ be subsets of vertices of $G$ (not necessarily disjoint).  An \emph{$(S_1, S_2)$-cycle} is a cycle in $G$ containing at least one vertex of $S_1$ and at least one vertex of $S_2$.  
Let $\dirG$ be an arbitrary orientation of $G$ and $e_1, \dots, e_m$ be an enumeration of $E(\dirG)$. Define $\g: E(\dirG) \to \Z \oplus \Z$ by 
$\g(e_i)=(\delta_1(e_i) 2^i, \delta_2(e_i)2^i)$, where $\delta_j(e_i)=1$ if $e_i$ has at least one end in $S_j$, and  $\delta_j(e_i)=0$ if $e_i$ does not have an end in $S_j$. 
As in the case of $S$-cycles, it is easy to check that $(\dirG, \g)$ is \sparse{}.
Since every $(S_1, S_2)$-cycle in $G$ corresponds to a $(\G_1,\G_2)$-non-zero cycle in $(\dirG, \g)$, we obtain the following theorem.

\begin{theorem}[$S_1$-$S_2$-cycles] \label{thm: s1s2 cycles}
There exists a function $f:\N\to\N$ such that for every graph $G$, every $S_1,S_2 \subseteq V(G)$, and every positive integer $k$, either $G$ has $k$ disjoint $(S_1, S_2)$-cycles or there exists a set of at most $f(k)$ vertices intersecting every $(S_1, S_2)$-cycle in $G$.
\end{theorem}

As we have seen above, robustness is a very useful condition with many applications.  However, the `real' theorem we prove is Theorem~\ref{doubleEscher}, which provides canonical obstructions to the \EP{} property for $(\G_1,\G_2)$-non-zero cycles.  Theorem~\ref{doubleEscher} follows fairly straightforwardly from our refined Flat Wall Theorem.  

There are further applications of Theorem~\ref{doubleEscher}.
For example, we prove that cycles not homologous to zero have the \emph{full} \EP{} property for graphs embedded on an \emph{orientable} surface (Corollary~\ref{orientable homology}).
As we cannot derive these results directly from Theorem~\ref{thm: EP hi cycles} or~\ref{thm: EP groups sparse}, 
we discuss these applications in Section~\ref{sec:deriving}.

It may be possible to prove \EP{} type results for cycles satisfying more than two constraints by extending our results to $\bigoplus_{i=1}^t \G_i$-labeled graphs. 
Let $(\dirG, \g)$ be a $\bigoplus_{i=1}^t \G_i$-labeled graph.  A cycle is \emph{$(\G_1, \dots, \G_t)$-non-zero} if it is $\G_i$-non-zero for all $i$. The following conjecture would imply the half-integral \EP{} property for cycles satisfying multiple constraints.

\begin{conjecture} \label{conj: main}
Fix $t \ge 1$ a positive integer and let $\G = \bigoplus_1^t \G_i$ where each $\G_i$ is a group. 
Then the set of all $\bigoplus_1^t \G_i$-labeled graphs has the half-integral \EP\ property for $(\G_1, \dots, \G_t)$-non-zero cycles.
Moreover, the \EP-function does not depend on the choice of $\G_1, \dots, \G_t$.
\end{conjecture}

\textbf{Remark.} Let $S_1, S_2$ and $S_3$ be subsets of vertices of a graph $G$.  An \emph{$(S_1,S_2,S_3)$-cycle} is a cycle in $G$ that uses at least one vertex from each of $S_1, S_2$ and $S_3$.  Note that the (full) \EP{} property does not hold for $(S_1,S_2,S_3)$-cycles.  To see this, let $G$ be a large grid, $S_1$ be the vertices on the top row, $S_2$  be the vertices on the rightmost column, and $S_3$ be the vertices on the bottom row.  Therefore, the obvious generalization of robustness does not guarantee the (full) \EP{} property for $(\G_1, \dots, \G_t)$-non-zero cycles if $t\geq 3$.  This shows that in some sense, Conjecture \ref{conj: main} is best possible.

\section{Preliminaries}\label{sec:background}

In this section, we introduce a number of concepts and notation we will need going forward.  Given two graphs $G$ and $H$, we denote by $G \cup H$ the graph with vertex set $V(G) \cup V(H)$ and edge set $E(G) \cup E(H)$.  
Analogously, we denote by $G \cap H$ the graph with vertices $V(G) \cap V(H)$ and edges $E(G) \cap E(H)$.  

\subsection{Group-labeled graphs}

In the introduction we claimed that if $C_1$ and $C_2$ are cycles such that $E(C_1)=E(C_2)$, then $C_1$ is $\G$-zero if and only if $C_2$ is $\G$-zero.   
While this is easy to see if $\G$ is abelian, this is not obvious for non-abelian groups, so we give a proof now.
Let $C=v_0 e_1 v_1 e_2 v_2 \dots e_\ell v_\ell$, where $v_0=v_\ell$.  Suppose that $C_1$ is $\G$-zero. Observe that
\begin{align*}
    0 &= \gamma(e_1, v_1)+ \dots + \gamma(e_\ell, v_\ell) \\
      &= -(\gamma(e_1, v_1)+ \dots + \gamma(e_\ell, v_\ell)) \\
      &= ((-\gamma(e_\ell, v_\ell)) + \dots + (-\gamma(e_1, v_1))) \\
      &= \gamma(e_\ell, v_{\ell-1}) + \dots + \gamma(e_1, v_0).
\end{align*}

Thus, the cycle starting at $v_0$ but with the opposite orientation as $C$ also has group-value $0$.  Furthermore, 
\begin{align*}
    0 &= \gamma(e_1, v_1)+ \dots + \gamma(e_\ell, v_\ell) \\
      &= -\gamma(e_1, v_1)+ (\gamma(e_1, v_1)+ \dots + \gamma(e_\ell, v_\ell)) + \gamma(e_1,v_1) \\
      &= \gamma(e_2, v_2) + \dots + \gamma(e_\ell, v_\ell)+\gamma(e_1,v_1).
\end{align*}
Thus, the cycle with the same orientation as $C$ but starting at $v_1$ also has group-value $0$.  By induction, every cycle with the same edge set as $C$ also has group-value $0$, as claimed.

In a slight abuse of notation, we will use $C$ to also refer to the graph with vertex set $V(C)$ and edge set $E(C)$, depending on the context.  

Let $v \in V(G)$, $\alpha \in \Gamma$, and $(\overrightarrow{G}', \gamma')$ be the $\Gamma$-labeled graph obtained from $(\overrightarrow{G}, \gamma)$ by adding $\alpha$ (on the right) to the labels
of the edges with head $v$ and adding $-\alpha$ (on the left) to the labels of the edges with tail $v$.  We refer to this operation as
\emph{shifting at $v$} (\emph{by $\alpha$}).  
A $\G$-labeled graph that can be obtained from $(\overrightarrow{G}, \gamma)$ via a sequence of shifts is said to be
\emph{shifting-equivalent} to $(\overrightarrow{G}, \gamma)$.  
We note that a key property of shifting is that it does not change the set of $\G$-non-zero cycles (if $\G$ is abelian, then the group-value of each cycle is actually unaltered by shifting).
In addition, the robustness of a group-labeling is maintained when performing shifts.

\subsection{Minors in graphs and the Flat Wall Theorem}

A \emph{separation} in a graph $G$ is a pair $(A, B)$ where $A$ and $B$ are edge-disjoint subgraphs of $G$ such that $A \cup B = G$.  The \emph{order} of a separation $(A, B)$ is $|V(A) \cap V(B)|$.  
We say that $(A,B)$ is a \emph{$k$-separation} if it has order at most $k$.  
The separation is \emph{trivial} if $V(A) \subseteq V(B)$ or $V(B) \subseteq V(A)$.

Let $k \ge 1$ be a positive integer and let $G$ be a graph.  A \emph{tangle of order $k$} in $G$ is a set $\zT$ of $(k-1)$-separations $(A, B)$ which satisfy the following.
\begin{enumerate}[(T1)]
\item  For every $(k-1)$-separation $(A, B)$ of $G$, either $(A, B) \in \zT$ or $(B, A) \in \zT$;
\item $V(A) \neq V(G)$ for all $(A, B) \in \zT$; 
\item $A_1 \cup A_2 \cup A_3 \neq G$ for all $(A_1, B_1), (A_2, B_2), (A_3, B_3) \in \zT$.
\end{enumerate}

If $(A,B) \in \cT$, we call $A$ the \emph{$\cT$-small} side of the separation.  
A separation and a tangle in a group-labeled graph $(\dirG, \g)$ is simply a separation and a tangle in $G$, respectively.

An example of a tangle which we will use going forward is the following tangle defined on a clique.  

\begin{lemma}[\cite{RS91}] \label{cliquetangle}
Let $t=\lceil \frac{2n}{3} \rceil$ and $\zT$ be the set of all $(t-1)$-separations $(A,B)$ of $K_n$ such that $V(B)=V(K_n)$.  Then $\zT$ is a tangle.  
\end{lemma}

We now describe several additional tangle constructions.  Let $\zT$ be a tangle of order $k$ in $G$ and let $\ell \leq k$.  Define $\zT'$ to be the set of $\ell$-separations $(A,B)$ such that $(A,B) \in \zT$.  It is immediate that $\zT'$ is also a tangle, which we call a \emph{truncation} or \emph{restriction} of $\zT$.  Note that $\zT'$ is a truncation of $\zT$ if and only if $\zT' \subseteq \zT$.  

Let $G$ be a graph and $e \in E(G)$.  We let $G \setminus e$ be the graph with vertex set $V(G)$ and edge set $E(G) \setminus \{e\}$.  We say $G \setminus e$ is the graph obtained from $G$ by \emph{deleting $e$}.  On the other hand, $G / e$ is the graph obtained from $G$ by deleting $e$ and then identifying the endpoints of $e$.  We say $G / e$ is the graph obtained from $G$ by \emph{contracting $e$}.  If $x \in V(G)$, we let $G -x$ be the graph obtained from $G$ by deleting $x$ and deleting all edges that have $x$ as an endpoint.  Note that deletion and contraction of edges commute and do not depend on the order in which they are performed.  Thus, if $C$ and $D$ are disjoint subsets of edges, we let $G / C \setminus D$ be the graph obtained from $G$ by deleting all edges in $D$ and contracting all edges in $C$.   

A graph $H$ is a \emph{minor} of a graph $G$ if a graph isomorphic to $H$ can be obtained from $G$ by deleting edges, contracting edges, and deleting vertices.  Equivalently, $H$ is a minor of $G$ if 
there is a function $\pi$ from $V(H)\cup E(H)$ to subgraphs of $G$
such that
\begin{enumerate}[(i)]
    \item $\pi(v)$ is a tree and a subgraph of $G$ and $\pi(v)$ is disjoint from $\pi(u)$ for distinct $u, v \in V(H)$; and 
    \item $\pi(uv)$ is an edge in $G$ joining $\pi(u)$ and $\pi(v)$ for all $uv\in E(H)$.
\end{enumerate}
We say that $\pi$ is an $H$-\emph{model} in $G$.

We next describe how to obtain a tangle in a graph from a tangle in one of its minors.  Let $H$ be a minor of $G$ and suppose that $\zT$ is a tangle of order $k \geq 2$ in $H$.  Let $C$ and $D$ be disjoint subsets of edges and $X \subseteq V(G)$ be such that $(G / C \setminus D)-X$ is isomorphic to $H$. If $A$ is a subgraph of $G$ we let $C_A:=C \cap A$, $D_A:=D \cap A$ and $X_A:=X \cap A$.  
Define $\zT_H$ to be the set of $(k-1)$-separations $(A,B)$ of $G$ such that $((A /C_A \setminus D_A)-X_A ,(B /C_B \setminus D_B) - X_B) \in \zT$.  It follows (see \cite{RS91}) that $\zT_H$ is a tangle in $G$.   We say that $\zT_H$ is the tangle in $G$ \emph{induced} by $\zT$.  If $K$ is a clique-minor in $G$ we always let $\zT_K$ denote the tangle in $G$ induced by the tangle in $K$ from Lemma~\ref{cliquetangle}.

If $\zT$ is tangle of order $k$ in $G$ and $X$ is a subset of vertices of size at most $k-2$, then it is easy to show that there is
a unique block $U$ of $G-X$ such that $V(U) \cup X$ is not contained in any $\zT$-small side.  
We call $U$ the \emph{$\zT$-large block} of $G-X$.  

Tangles are extremely useful objects in graph structure theory.  
For example, they arise in a natural way when considering the \EP-property. Let $\Gamma_1$ and $\Gamma_2$ be  groups, $\cal{G}$ a set of $(\Gamma_1 \oplus \Gamma_2)$-labeled graphs, and $f:\N\to\N$ a function.  A pair $((\dirG, \gamma), k)$ is a \emph{minimal counterexample to $f$ being an \EP{} function} if it satisfies the following conditions:
\begin{enumerate}[(MC1)]
\item $(\dirG, \gamma)$ is an element of $\cal{G}$;
\item\label{item:MC2} there does not exist a set $X \subseteq V(G)$ with $|X| \le f(k)$ such that $(\dirG - X, \gamma)$ has no $(\G_1,\G_2)$-non-zero cycle, nor does $(\dirG, \gamma)$ contain $k$ disjoint $(\G_1,\G_2)$-non-zero cycles; 
\item for all $(\dirG', \gamma') \in \cal{G}$ and for all $k' < k$, the graph $(\dirG', \gamma')$ either has $k'$ disjoint $(\G_1,\G_2)$-non-zero cycles or there exists $X' \subseteq V(\dirG')$ with $|X'| \le f(k')$ such that $(\dirG' - X', \gamma')$ has no $(\G_1,\G_2)$-non-zero cycle.
\end{enumerate}
The definition for a pair $((\dirG, \gamma), k)$ being a minimal counterexample to $f$ being a half-integral \EP{} function is analogous.

The following result is Lemma 2.1 of \cite{Wol11}, rephrased in terms of $(\G_1,\G_2)$-group-labeled graphs.  
We include the proof for completeness.

\begin{lemma}[Lemma 2.1 \cite{Wol11}]\label{lemma: EP tangle}
Let $\Gamma_1$, $\Gamma_2$ be  groups and $\cal{G}$ a set of $(\Gamma_1 \oplus \Gamma_2)$-labeled graphs.  Let $f: \N \rightarrow \N$ be a function that is not an \EP{} function (not a half-integral \EP{}-function) for $(\G_1,\G_2)$-non-zero cycles.  Let $((\dirG, \gamma), k)$ be a minimal counterexample to $f$ being an \EP{} function (half-integral \EP{}-function). Let $t \in \mathbb{N}$ be such that $t \leq f(k) -2f(k-1)$ and $t \leq f(k) / 3$, and let $\zT$ be the set of all $(t-1)$-separations $(A,B)$ of $(\dirG, \gamma)$ such that $B$ contains a $(\Gamma_1, \Gamma_2)$-non-zero cycle.  Then $\zT$ is a tangle of order $t$ in $(\dirG, \gamma)$.
\end{lemma}

\begin{proof}
We prove the statement for when $f$ is not an \EP{} function; the proof for the case when $f$ is not a half-integral \EP{} function is identical.

Let $(A, B)$ be a $(t-1)$-separation in $(\dirG, \gamma)$.   We claim that exactly one of $A$ or $B$ contains a $(\G_1, \G_2)$-non-zero cycle.  If neither $A$ nor $B$ contains a $(\G_1, \G_2)$-non-zero cycle, then $V(A) \cap V(B)$ is a set intersecting every such cycle, a contradiction to (MC2).  
Assume now that both $A$ and $B$ contain a $(\G_1, \G_2)$-non-zero cycle.  
Then neither of the $(\G_1 \oplus \G_2)$-labeled graphs $A - V(B)$ and $B-V(A)$ contain $k-1$ disjoint $(\G_1, \G_2)$-non-zero cycles.  
By (MC3), 
there exist vertex sets $X_A$ and $X_B$ each of size at most $f(k-1)$ which intersect every $(\G_1, \G_2)$-non-zero cycle in $A - V(B)$ and $B-V(A)$, respectively.  Thus $X = X_A \cup X_B \cup (V(A) \cap V(B))$ is a set intersecting every $(\G_1, \G_2)$-non-zero cycle in $(\dirG, \gamma)$, a contradiction to (MC2).  
We conclude that exactly one of $A$ and $B$ contains a $(\G_1, \G_2)$-non-zero cycle, as claimed.

It follows that (T1)~and (T2)~hold for $\zT$.  To see that (T3)~holds as well, assume there exist separations $(A_1, B_1)$, $(A_2, B_2)$, and $(A_3, B_3)$ in $\zT$ with $(\dirG, \gamma)  = A_1 \cup A_2 \cup A_3$.  Since $A_i$ contains no $(\G_1, \G_2)$-non-zero cycle, every cycle in $A_1 \cup A_2 \cup A_3=(\dirG, \g)$ must meet $X := \bigcup_{i=1}^3 (V(A_i) \cap V(B_i))$. Since $X$  is of size at most $3t-3$, we contradict (MC2).  This completes the proof.   
\end{proof}

Theorem \ref{thm: main wall} is a refinement of the Flat Wall Theorem of Robertson and Seymour \cite[Theorem 9.8]{RS95}.  We first need some definitions and notation before we can state the Flat Wall Theorem.

\begin{figure}[t]
\centering
\begin{tikzpicture}[scale=0.7]
\def\height{1}
\def\width{1}
\def\sizeofwalla{6}
\def\sizeofwallb{6}

\foreach \i in {2,...,\sizeofwallb}{
	\draw[very thick] (0,\i-1)--(2*\sizeofwalla+1,\i-1);
}
\draw[very thick] (1,0)--(2*\sizeofwalla+1,0);
\draw[very thick] (0,\sizeofwallb)--(2*\sizeofwalla,\sizeofwallb);

\pgfmathsetmacro{\bh}{\sizeofwallb/2-1}
\foreach \i in {0,1,...,\sizeofwalla}{
	\foreach \j in {0,1,...,\bh}{
		\draw[very thick] (2*\i+1,2*\j)--(2*\i+1,2*\j+1);
	}
}

\foreach \i in {0,1,...,\sizeofwalla}{
	\foreach \j in {0,1,...,\bh}{
		\draw[very thick] (2*\i,2*\j+1)--(2*\i,2*\j+2);
	}
}

\foreach \i in {1,...,\sizeofwalla}{
	\draw[very thick] (2*\i-1,\sizeofwallb) circle (0.2);
}

\foreach \i in {0,1,...,\bh}{
	\draw[very thick,white,dashed] (3,2*\i)--(3,2*\i+1);
	\draw[very thick,white,dashed] (2,2*\i+1)--(2,2*\i+2);
}
\foreach \i in {2,...,\sizeofwallb}{
	\draw[very thick,white,dashed] (2,\i-1)--(3,\i-1);
}

\draw[very thick] (-0.2+1,-0.2) rectangle (1.2,0.2);
\draw[very thick] (-0.2+2*\sizeofwalla+1,-0.2) rectangle (0.2+2*\sizeofwalla+1,0.2);
\draw[very thick] (-0.2+2*\sizeofwalla,-0.2+\sizeofwallb) rectangle (0.2+2*\sizeofwalla,0.2+\sizeofwallb);
\draw[very thick] (-0.2,-0.2+\sizeofwallb) rectangle (0.2,0.2+\sizeofwallb);

\end{tikzpicture}
\caption{An elementary $6$-wall $W$.
The nails of $W$ on the top row are circled, the corners are marked by squares, and the second vertical path is marked dashed.}\label{fig: wall}
\end{figure}
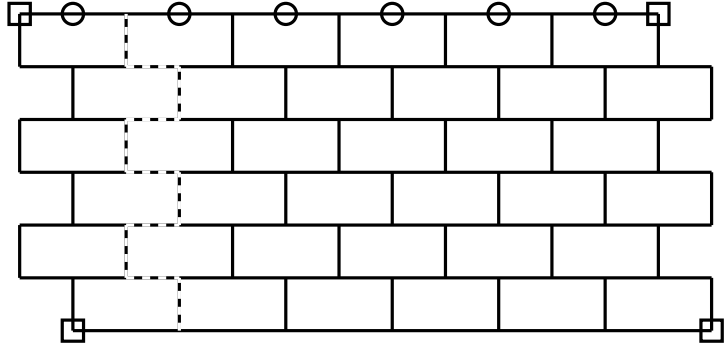

We first define an elementary $r$-wall.  An \emph{elementary $6$-wall} is shown in Figure~\ref{fig: wall}.  
Let  $r,s\ge2$ be an  integer.  An $r\times s$-grid is the graph with vertex set $[r]\times[s]$ in which
$(i,j)$ is adjacent to $(i',j')$ if and only if $|i-i'|+|j-j'|=1$.
An \emph{elementary $r$-wall} is obtained from the  
$2(r+1)\times (r+1)$-grid
by deleting all edges with ends $(2i-1,2j-1)$ and $(2i-1,2j)$
for all $i\in[r]$ and $j\in [\lfloor r/2\rfloor]$
and all edges with ends  $(2i,2j)$ and $(2i,2j+1)$
for all $i\in [r]$ and $j\in [\lfloor (r-1)/2\rfloor]$
and then deleting the two resulting vertices of degree 1.
Every facial cycle of a finite face has length $6$, and is called a \emph{brick}.  
The four vertices that are the intersection of a leftmost or rightmost vertical path
with a topmost or bottommost horizontal path are called the 
 \emph{corners} of the wall. We denote the vertices of degree $2$ which are not corners as the \emph{nails} of the wall.

A subdivision of an elementary $t$-wall is called a \emph{$t$-wall} or simply a \emph{wall}.  The \emph{bricks} of a wall are the subdivided 6-cycles corresponding to the bricks of the elementary wall.  Similarly, the \emph{corners} and \emph{nails} of a wall are the vertices corresponding to the corners and nails of the elementary wall before subdividing edges.   

Sometimes we tacitly assume that a wall is embedded in the plane as shown in Figure~\ref{fig: wall}.  The \emph{outercycle} or \emph{boundary cycle} of the wall is the facial cycle of the infinite face.   Let $W$ be a $t$-wall.  Let $v_1$, $v_2$, $v_3$, and $v_4$ be the four corners of the wall as they occur in that clockwise order on the boundary cycle $C$ with $v_1$ the vertex corresponding to the corner $(1, 2)$ of the elementary $t$-wall.  
Let $P_0^{(h)}$ be the subpath of $C$ with endpoints $v_1$ and $v_2$ which is disjoint from $\{v_3, v_4\}$ and let $P_{t}^{(h)}$ be the subpath of $C$ with ends $v_3$ and $v_4$ which is disjoint from $\{v_1, v_2\}$.  Observe that there is a unique set of disjoint paths $P_{0}^{(v)}, \dots, P_{t}^{(v)}$ such that each path has one end in $V(P_0^{(h)})$, the other end in $V(P_t^{(h)})$, and no other vertex in $V(P_0^{(h)} \cup P_t^{(h)})$.  We call $P_{0}^{(v)}, \dots, P_{t}^{(v)}$ the \emph{vertical paths} of $W$.  
There is also a unique set of disjoint paths $P_{1}^{(h)}, \dots, P_{t-1}^{(h)}$, where each has one end in $P_{0}^{(v)}$, the other end in $P_{t}^{(v)}$, and no other vertex in $P_{0}^{(v)} \cup P_{t}^{(v)} \cup P_0^h \cup P_t^h$.  The paths $P_0^{(h)}, \dots, P_t^{(h)}$ are the \emph{horizontal paths} of the wall.  Note that if we are just given $W$ as a graph, then the nails and corners are not necessarily uniquely defined; in such cases, we assume that the nails and corners are arbitrarily chosen.

Next, we define a subwall of a $t$-wall $W$. Let $s\leq t$. 
 An \emph{$s$-subwall} of $W$ is a subgraph $W'$ of $W$ such that $W'$ is an $s$-wall and there exists a choice of corners of $W'$ so that every horizontal path of $W'$ is a subpath of a horizontal path of $W$ and every vertical path of $W'$ is a subpath of a vertical path of $W$.  
Let $I^{(h)},I^{(v)}\subseteq  \{0, \dots, t\}$ with $|I^{(h)}|=|I^{(v)}|=s+1$ and such that every horizontal path of $W'$ is a subpath of $P^{(h)}_j$ for some $j \in I^{(h)}$ and every vertical path of $W'$ is a subpath of $P^{(v)}_j$ for some $j \in I^{(v)}$.  
We say that $W'$ is \emph{$k$-contained} in $W$ if $\min \{I^{(h)},I^{(v)}\}\geq k$ and $\max \{I^{(h)},I^{(v)}\}\leq t-k$.

We observed that there may be several possible choices for the set of nails of a wall $W$.  
However, if $W'$ is a 1-contained subwall of $W$, there are choices of nails $N'$ for $W'$ which are more natural than others;  
that is, we assume that
$$N'\subseteq\{v\in V(C'): d_{W'}(v)=2 \text{ and } d_{W}(v)=3\} \setminus \{v_1', \dots, v_4'\},$$
where $C'$ is the outercycle of $W'$ and $v_1', \dots, v_4'$ are the corners of $W'$.  
We say that $N'$ is the set of \emph{nails of $W'$ with respect to $W$}.

Let us turn to the definitions needed for the Flat Wall Theorem.

\begin{definition}
Let $G$ be a graph, $X \subseteq V(G)$, and $(C,D)$ be a separation of $G$ such that 
\begin{enumerate}[(i)]
\item 
$|V(C \cap D)| \leq 3$, 
\item
$X \subseteq V(C)$, and
\item
there is a family of $|V(C \cap D)|$ paths from a vertex $d \in D - C$  to $X$ that are disjoint except for $d$.  
\end{enumerate}
Let $H$ be the graph obtained from $C$ by adding exactly those edges such that $C \cap D$ is a clique.  
We say that $H$ is an \textbf{elementary $X$-reduction} of $G$ (with respect to $(C,D)$).
An \textbf{$X$-reduction} of $G$ is a graph that can be obtained from $G$ via a sequence of elementary $X$-reductions. 
\end{definition}

\begin{definition}\label{def: flat wall}
Let $G$ be a graph and let $W$ be a wall in $G$ with outercycle $O$.  Let $(A,B)$ be a separation of $G$ such that
\begin{enumerate}[(i)]
\item $V(A \cap B) \subseteq V(O)$, and $V(W) \subseteq V(B)$, 
\item there is a choice of nails and corners of $W$ such that every nail and corner is in $A$, and 
\item there is a $V(A \cap B)$-reduction $G_0$ of $B$ such that $G_0$ can by drawn in a disk $\Delta$ with all vertices of $A \cap B$ drawn on the boundary
of $\Delta$ according to their order on $O$.  
\end{enumerate}
In this case we say that the wall $W$ is \textbf{flat} in $G$.  The separation $(A, B)$ \textbf{certifies} that $W$ is flat.
\end{definition}

Let $W$ be a $t$-wall.  We now describe a tangle $\zT$ of order $t+1$ inside $W$ as follows.  Let $\zT$ consist of the set
of all $t$-separations $(A,B)$ such that $B$ contains a horizontal path.  The set $\zT$ 
is indeed a tangle (see \cite{RS91}).  If $W$ is a minor of $G$, we always let $\zT_W$ denote the tangle in $G$ induced by the tangle in $W$ described above.  We can now state the Flat Wall Theorem.  

\begin{theorem}[Theorem 1.5 \cite{KTW12}] \label{flatwalltheorem}
Let $r,t \in \mathbb{N}, R:=49152t^{24}(40t^2+r)$, $G$ be a graph, and $W$ be an $R$-wall in $G$.  Then either $G$ has a $K_t$-minor $K$ such that
$\zT_K$ is a restriction of $\zT_W$, or there exists a subset $Z$ of vertices of $G$ of size at most $12288t^{24}$ and an $r$-subwall $W'$ of $W$ such that
$V(W') \cap Z$ is empty and $W'$ is flat in $G-Z$. 
\end{theorem}

Thus, loosely speaking, every large wall $W$ in a graph $G$ can be used to construct a large clique minor, or alternatively, we can delete a bounded size set of vertices such that a large subwall of $W$ induces a planar subgraph of $G$ up to $1$-, $2$-, and $3$-sums.  

Observe that a subwall of a flat wall is once again flat and thus we also speak about \emph{flat subwalls}.

Let $(\dirG, \g)$ be a $\G$-labeled graph. 
We abuse terminology slightly and call a subgraph $(\dirW, \g)$ of $(\dirG, \g)$  a \emph{wall} if $W$ is a wall in $G$.  
We naturally extend all terminology for walls in graphs to
walls in group-labeled graphs.  

We say that $(\dirG, \g)$ is \emph{null-labeled} if for all $e \in E(\dirG)$, we have $\g(e)=0$.  
We say that $(\dirG, \g)$ is \emph{$\G$-bipartite} if every cycle of $(\dirG, \g)$ is $\G$-zero. If $\G=\G_1 \oplus \G_2$ , then $(\dirG, \g)$ is \emph{$\G_i$-bipartite} if every cycle of $(\dirG, \g_i)$ is $\G_i$-zero.

We will repeatedly use the following basic lemma \cite{GG09} without explicit reference. 

\begin{lemma}[\cite{GG09}]
Let $\Gamma$ be a group and $(\dirG, \g)$ be a $\Gamma$-labeled graph.  If $(\dirG, \g)$ is $\G$-bipartite, then $(\dirG, \g)$ is
shifting-equivalent to a null-labeled graph.  
\end{lemma}

We say that a $\G$-labeled graph $(\dirH, \g_H)$ is a \emph{minor} of another $\G$-labeled graph $(\dirG, \g_G)$ if $(\dirH, \g_H)$ can be obtained from $(\dirG, \g_G)$ via any combination of
edge deletions, vertex deletions, shifts, and contracting null-labeled edges.  Note that the definition of contracting an edge does not change the orientations of any other edge. In other words, if we contract a null-labeled edge $e$ to obtain a new vertex $v_e$, and $f$ is an edge in $(\dirG, \g)$ with $\tail_{\dirG}(f)$ (respectively, $\head_{\dirG}(f)$) equal to an end of $e$ in $\dirG$, then $\tail_{\dirG / e}(f)=v_e$ (respectively, $\head_{\dirG / e}(f)=v_e)$).

\subsection{Graphs and homology}

As promised, we finish this section by proving Proposition~\ref{homologylabel}.
\begin{proof}[Proof of Proposition~\ref{homologylabel}]
Let $G$ be a graph embedded on a surface $\Sigma$ and $\G$ be the homology group of $\Sigma$. We proceed by induction on $|E(G)|$.  For a closed walk $W$ in $G$, we let $h(W)$ denote the homology class of $W$.  Let $e=uv \in E(G)$.  By induction, there is an orientation $\overrightarrow{G-e}$ of $G-e$ and a labeling $\g:E(\overrightarrow{G-e}) \to \G$ such that $\g(W)$ is the homology class of $W$ for all closed walks $W$ in $G-e$. We orient $uv$ from $u$ to $v$.  If there is no path from $v$ to $u$ in $G-e$ we let $\g(uv)$ be an arbitrary element of $\G$.  
Otherwise, we choose a path $P$ in $G-e$ from $v$ to $u$, and we let $C$ be the cycle in $G$ obtained by following $P$ and then $uv$.  We then define $\g(uv)$ to be $h(C)-\g(P)$. Suppose we are in the first case.  Let $W$ be a closed walk in $G$.  We may assume that $uv \in W$, otherwise there is nothing to prove. 
We may also assume that $W$ begins and ends at $u$.  Thus, $W=uvW_1vuW_2uv \dots W_{2k-1}vuW_{2k}$, where each $W_i$ is a closed walk in $G-uv$ (possibly $W_i$ is a walk of length $0$).  By induction, we have $\g(W_i)=h(W_i)$ for all $i\in[k]$.  Finish by observing that $\g(W)=\sum_{i=1}^{2k} \g(W_i)=\sum_{i=1}^{2k} h(W_i)=h(W)$.  In the second case, let $W$ be a closed walk in $G$ with $uv \in W$.  It suffices to consider the case that $W=uvQ$, where $Q$ is a walk from $v$ to $u$ in $G-uv$.  Let $W'=(uvP)(P^{-1}Q)$, where $P$ is the path used to define $\g(uv)$.  Since $P^{-1}Q$ is a walk in $G-uv$, by induction we have $h(P^{-1}Q)=-\g(P)+\g(Q)$.  Thus,
$
h(W)=h(W')=h(uvP)+h(P^{-1}Q)=\g(uv)+\g(P)-\g(P)+\g(Q)=\g(W).
$
\end{proof}

\section{Obstructions} \label{sec:obstructions}

In this section we introduce some canonical counterexamples to the \EP{} property for $(\G_1,\G_2)$-non-zero cycles that appear in Theorem \ref{doubleEscher}.  

In his proof of Theorem \ref{thm: reed}, Reed \cite{Ree99} proves a significantly stronger statement.  
He defines a set of canonical counterexamples to the \EP~property for odd cycles called \emph{Escher-walls} and then shows that any graph either has many disjoint odd cycles, or a bounded set of vertices intersecting every odd cycle, or it contains a large Escher-wall.  
An \emph{Escher-wall of height $h$}, 
is a graph obtained from a bipartite $h$-wall $W$ by adding $h$ disjoint paths $P_1, \dots , P_h$ such that for all $i \in [h]$,

\begin{enumerate}[(i)]
\item 
$P_i$ has one end in the $i$-th brick of the top horizontal path of $W$ and the other end in the
$(h-i+1)$-th brick of the bottom horizontal path of $W$, but otherwise contains no other vertices of $W$, and
\item
$W \cup P_i$ contains an odd cycle.
\end{enumerate}

It is easy to show that an Escher-wall $G$ of height $h$ does not contain two disjoint odd cycles nor a set $X \subseteq V(G)$ with $|X| < h$ and $G-X$ bipartite.  Reed \cite{Ree99} showed that Escher-walls
are the only obstructions to the \EP{} property for odd cycles. 

\begin{theorem}[\cite{Ree99}] \label{escherwalls}
There exists a function $f:\N\to\N$
such that for every graph $G$ and every integer $k$,
\begin{enumerate}[(i)]
	\item $G$ contains $k$ disjoint odd cycles,
	\item there exists a set $X$ such that $G-X$ is bipartite and $|X|\leq f(k)$, or
	\item $G$ contains an Escher-wall of height $k$.
\end{enumerate}
\end{theorem}

We will prove a similar result (Theorem \ref{doubleEscher}) for $(\G_1,\G_2)$-non-zero cycles in Section \ref{sec:deriving}, which in fact will imply Theorem \ref{escherwalls}. 
 
We begin with some definitions.  
Let $t \ge 1$ be an integer and $W$ be a $t$-wall with a set of nails $N$.  Let $P_j^{(h)}$ and $P_j^{(v)}$ for $0 \le j \le t$ be the horizontal and vertical paths of $W$.  For any subwall $W'$ of $W$ with set of nails $N'$, let $j$ be the minimum index such that $P_j^{(h)}$ intersects $W'$, and let $P$ be the horizontal path of $W'$ contained in $P_j^{(h)}$.  We define the \emph{top nails} of $W'$ to be the linearly ordered set of vertices $\overline{N}$ consisting of the nails of $W'$ contained in $P$ such that $x, y \in \overline{N}$ satisfy $x \le y$ if and only if we encounter $x$ before $y$ when traversing $P_j^{(h)}$ from its endpoint in $P_0^{(v)}$ to its endpoint in $P_t^{(v)}$.

Let $G$ be a graph.  
For $X,Y \subseteq V(G)$ an \emph{$X$--$Y$-path} is a path $P$ with at least one edge such that one end of $P$ is in $X$ the other end is in $Y$, and $V(P)$ is otherwise disjoint from $X \cup Y$. 
An \emph{$X$-path} is an $X$--$X$-path.
An \emph{$X$-linkage} (or simply a \emph{linkage}) is a family of disjoint $X$-paths in $G$.  For a subgraph $H$ of $G$, a \emph{$H$-linkage} is a set of pairwise disjoint $V(H)$-paths which are edge-disjoint from $H$.

Suppose in addition that the vertex set $X$ is a linearly ordered set.  
If $P$ is an $X$-path, we call the smaller endpoint of $P$ the \emph{left endpoint}, and the larger endpoint of $P$ the \emph{right endpoint}.  
Let $x,y\in X$ with $x\leq y$. 
We let $[x,y]$ denote the set of all $z\in X$ such that $x\leq z\leq y$, and call $[x,y]$ an \emph{interval}.  
We write $[x,y] < [z,w]$ if $y<z$.  
For a family $\cP$ of disjoint $X$-paths, 
we define $I_{\cP}\subseteq X$ as the minimal interval (under inclusion) containing all endpoints of paths in $\cP$. We let  
$I_{\cP}^{\ell}$ and $I_{\cP}^{r}$ be the minimal intervals containing all the left and right endpoints of $\cP$, respectively.

In the following we tacitly assume that $X$-paths are always traversed from their left endpoint to their right endpoint.  Let
 $(p_1,p_2)$ and $(q_1,q_2)$ be pairs of elements of $X$ such that $p_1<q_1$, $p_1<p_2$, and $q_1<q_2$.
We say that $(p_1,p_2)$ and $(q_1,q_2)$ are \emph{in series} if $p_2<q_1$;
\emph{nested} if $p_1 < q_1 < q_2 < p_2$;
or \emph{crossing} if $p_1 < q_1 < p_2 < q_2$.

Let $P$ and $Q$ be disjoint $X$-paths with ends $(p_1,p_2)$ and $(q_1,q_2)$, respectively. 
We say that $P$ and $Q$ are \emph{in series, nested,} or \emph{crossing}, according as $(p_1,p_2)$ and $(q_1,q_2)$ are in series, nested, or crossing.
A collection $\cP$ of $X$-paths is \emph{in series, nested,} or \emph{crossing} if all pairs of paths in $\cP$ are in series, nested, or crossing, respectively.  We say that $\cP$ is \emph{pure}
if $\cP$ is nested, crossing, or in series. Thus, there are three \emph{types} of pure linkages. 
See Figure~\ref{fig: linkage} for illustration.
\begin{figure}[t]
\centering
\begin{tikzpicture}[scale=0.7]
\def\height{1}
\def\width{1}

\begin{scope}[shift={(9,0)}]
\draw[very thick] (-1,0)--(7,0);

\draw[very thick] (0,0)..controls (1,3) and (5,3) .. (6,0);
\draw[very thick] (1,0)..controls (2,2) and (4,2) .. (5,0);
\draw[very thick] (2,0)..controls (2.6,1) and (3.4,1) .. (4,0);

\foreach \i in {0,1,2,4,5,6}{
	\fill (\i,0) circle (0.18);
}	
\end{scope}

\begin{scope}
\draw[very thick] (-1,0)--(7,0);

\draw[very thick] (0,0)..controls (1,2) and (3,2) .. (4,0);
\draw[very thick] (1,0)..controls (2,2) and (4,2) .. (5,0);
\draw[very thick] (2,0)..controls (3,2) and (5,2) .. (6,0);

\foreach \i in {0,1,2,4,5,6}{
	\fill[very thick] (\i,0) circle (0.18);
}	
\end{scope}

\begin{scope}[shift={(5,-3)}]
\draw[very thick] (-1,0)--(6.2,0);

\draw[very thick] (0,0)..controls (0.3,1) and (0.9,1) .. (1.2,0);
\draw[very thick] (2,0)..controls (2.3,1) and (2.9,1) .. (3.2,0);
\draw[very thick] (4,0)..controls (4.3,1) and (4.9,1) .. (5.2,0);

\foreach \i in {0,1.2,2,3.2,4,5.2}{
	\fill[very thick] (\i,0) circle (0.18);
}	
\end{scope}

\end{tikzpicture}
\caption{Linkages of size $3$ that are crossing, nested, and in series, respectively.
The six endpoints of the paths are linearly ordered from left to right.}\label{fig: linkage}
\end{figure}
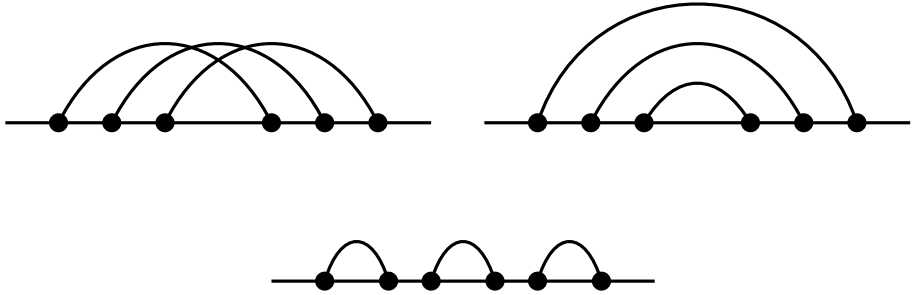

We now describe our set of obstructions. 
Let $\G_1$ and $\G_2$ be non-trivial  groups.  
Let $h$ be a positive integer and let $(\dirG, \g)$ be a $(\G_1 \oplus \G_2)$-labeled graph  be such that
\begin{enumerate}[(i)]
\item 
$E(G)=E(W) \cup E(\cP) \cup E(\cQ)$, where $W$ is a $4h$-wall, and $\cP$ and $\cQ$ are pure $W$-linkages of size 
$h$ such that the endpoints of the paths in $\cP$ or  $\cQ$ are vertices in the topmost horizontal path of $W$ of degree $2$ in $W$, 
no two of which lie in the same brick of $W$,
and are not corners of $W$,
\item $(\dirW, \g)$ is null-labeled,
\item $\cP\cup \cQ$ is a linkage of size $2h$,
\item $\cP$ and $\cQ$ are of different type,
\item $\g_1(P)\not=0$ and $\g_2(P)=0$ for all $P\in\cP$,
\item $\g_1(Q)=0$ and $\g_2(Q)\not=0$ for all $Q\in\cQ$,
	\item if $\cP$ is crossing or nested, then $\g_1(P_1)=\g_1(P_2)$ for all $P_1,P_2\in\cP$,
	\item if $\cQ$ is crossing or nested, then $\g_2(Q_1)=\g_2(Q_2)$ for all $Q_1,Q_2\in\cQ$,
	\item If $I_\cP \cap I_\cQ = \emptyset$, then $I_\cP < I_\cQ$ and otherwise if $I_\cP \cap I_\cQ \neq \emptyset$, then $I^\ell_\cP < I^\ell_\cQ < I^r_\cP < I^r_\cQ$ and neither $\cP$ nor $\cQ$ is in series.
\end{enumerate}

It is not difficult to see that $(\dirG, \g)$ does not contain two disjoint $(\G_1,\G_2)$-non-zero cycles.
However, one needs more than $h$ vertices to cover all $(\G_1,\G_2)$-non-zero cycles in $(\dirG, \g)$.   

\textbf{Remark.} Note that if $\cP$ is nested and $\cQ$ is in series, then $(\dirG, \g)$ is planar.  
Thus, $(\G_1,\G_2)$-non-zero cycles do not have the (full) \EP{} property even when restricted to $(\G_1 \oplus \G_2)$-labeled graphs $(\dirG, \g)$ for which $G$ is planar.

\section{$\G$-Odd $K_t$-models} \label{sec:oddkt}

Let $\G$ be a group and $(\dirG, \g)$ be a $\G$-labeled graph.  A \emph{$K_t$-model} $\pi$ in $(\dirG, \g)$ is simply a $K_t$-model in $G$. 
That is, $\pi$ is a function mapping $V(K_t)\cup E(K_t)$ to subgraphs of $(\dirG, \g)$ such that
\begin{enumerate}[(i)]
	\item $\{\pi(v)\}_{v\in V(K_t)}$ is a set of pairwise disjoint trees in $(\dirG, \g)$, and
	\item for all $uv\in E(K_t)$, $\pi(uv)$ is an edge of $(\dirG, \g)$ joining $\pi(u)$ and $\pi(v)$.  
\end{enumerate}
A $K_t$-model in $(\dirG, \g)$ is \emph{$\G$-odd} if it additionally satisfies
\begin{enumerate}
	\item[(iii)] for all distinct $x,y,z\in V(K_t)$,
	the unique cycle of $(\dirG, \g)$ in $\pi(x)\cup \pi(y)\cup \pi(z)\cup \pi(xy)\cup \pi(xz)\cup \pi(yz)$ is $\G$-non-zero.
\end{enumerate}
Observe that if $\G=\mathbb{Z}_2$, then $(\dirG, \g)$ has a $\G$-odd $K_t$-model if and only if $(\dirG, \g)$ has an odd-$K_t$-minor in the sense of \cite{GGRSV09}.

Next, we extend this definition to the case when $G$ is a $(\G_1 \oplus \G_2)$-labeled graph. 
A \emph{$(\G_1,\G_2)$-odd $K_t$-model} in $(\dirG, \g)$ is a function $\pi$ mapping $V(K_t)\cup E(K_t)$ to subgraphs of $(\dirG, \g)$
such that 
\begin{enumerate}[(i)]
	\item $\{\pi(v)\}_{v\in V(K_t)}$ is a set of pairwise disjoint trees in $(\dirG, \g)$,
	\item for all $uv\in E(K_t)$, $\pi(uv)$ is a set of either one or two edges of $(\dirG, \g)$
	such that for all $e\in \pi(uv)$, $e$ joins $\pi(u)$ and $\pi(v)$, and
	\item for all $i \in [2]$ and distinct $x,y,z\in V(K_t)$,
	there are $e_{xy}\in \pi(xy),e_{xz}\in \pi(xz),e_{yz}\in \pi(yz)$
	such that the unique cycle of $\pi(x)\cup \pi(y)\cup \pi(z)\cup e_{xy}\cup e_{xz}\cup e_{yz}$ is $\G_i$-non-zero.
\end{enumerate}

Let $\pi$ be a  $K_t$-model in $(\dirG, \g)$.  We say $V(K_t)$ and $E(K_t)$ are the \emph{vertices} and \emph{edges} of $\pi$, respectively. Recall that there is a tangle of order $\lceil \frac{2t}{3} \rceil$ in $K_t$ (see Lemma~\ref{cliquetangle}). Thus, $\pi$ induces a tangle of order $\lceil \frac{2t}{3} \rceil$ in $(\dirG, \g)$, which we denote as $\cT_\pi$. 
We say $\pi'$ is an \emph{enlargement} of $\pi$
if $\pi'$ is a $K_{t'}$-model in $(\dirG, \g)$,
and for all $v\in V(K_{t'})$, the tree $\pi'(v)$ contains some $\pi(u)$ for $u\in V(K_t)$ (and thus $t'\leq t$).

The following lemma will be used frequently in our proofs. 

\begin{lemma}[\cite{CGGGLS06}]\label{lemma: non-zero A-paths}
Let $\G$ be a  group and $(\dirG, \G)$ be a $\G$-labeled graph. For every $A \subseteq V(G)$, one of the following two statements hold.
\begin{enumerate}[(i)]
	\item there are $k$ pairwise disjoint $\G$-non-zero $A$-paths in $(\dirG, \g)$, or 
	\item there is a set $X$ of at most $2k-2$ vertices such that $(\dirG-X, \g)$ does not contain a $\G$-non-zero $A$-path.
\end{enumerate}
\end{lemma}

Let $R(n,m)$ be the least number such that for every $\{\text{red}, \text{blue}\}$-colouring of the edges of the $3$-uniform complete hypergraph on $R(n,m)$ vertices, there
is always a red $3$-uniform complete hypergraph on $n$ vertices or a blue $3$-uniform complete hypergraph on $m$ vertices.
It is well known that $R(n,m)$ is finite for all positive integers $n,m$.

\begin{lemma} \label{lemma: odd Kt minor}
Let $\G$ be a group and $(\dirG, \g)$ be a $\G$-labeled graph.  
For every $K_{R(t,24t)}$-model $\pi$ in $(\dirG, \g)$, there is either a $\G$-odd $K_t$-model in $(\dirG, \g)$ that is an enlargement of $\pi$,
or there is a set of vertices $X$ such that $|X|<8t$ and the $\zT_\pi$-large block of $(\dirG-X, \g)$ is $\G$-bipartite.
\end{lemma}

\begin{proof}
In view of the statement, we may assume that $t\geq 3$.
Let $\pi$ be a $K_{R(t,24t)}$-model in a $\G$-labeled graph $(\dirG, \g)$.
Let $H$ be a $3$-uniform complete hypergraph with vertex set $\bigcup_{u \in V(K_{R(t,24t)})}\pi(u)$.
Color an edge $xyz$ of $H$ red, if 
the cycle of $(\dirG, \g)$ in 
\begin{align*}
    \pi(x)\cup \pi(y)\cup\pi(z)\cup\pi(xy)\cup\pi(xz)\cup\pi(yz);
\end{align*}
is $\G$-non-zero;
otherwise, color $xyz$ blue.  
By the definition of the Ramsey number $R(t,24t)$, 
there is a $\G$-odd $K_t$-model in $(\dirG, \g)$ or a $\G$-bipartite $K_{24t}$-model $\eta$ in $(\dirG, \g)$ that is an enlargement of $\pi$.

Henceforth, we assume the latter case.
Let $\{v_1,\ldots,v_{24t}\}$ be the vertices of $\eta$.
Since $\eta$ is $\G$-bipartite, by shifting we may assume that all edges of $\eta$ are null-labeled
(the definition of a $\G$-odd $K_t$-model is invariant under shifting).  

Next we show the following claim.
\begin{equation}\label{parity breaking paths}
\begin{minipage}[c]{0.8\textwidth}\em
Suppose there are $t$ disjoint $\G$-non-zero paths $P_1,\ldots,P_t$ in $(\dirG, \g)$ such that for all $i \in [t]$, 
the path $P_i$ is a $\eta(v_{2i-1})$--$\eta(v_{2i})$-path
and $\eta(v_j)\cap P_i=\es$ for all $j \in[2t]\sm \{2i-1,2i\}$,
then there is a $\G$-odd $K_t$-model $\eta'$ in $(\dirG, \g)$ which is an enlargement of $\eta$ (and thus of $\pi$).
\end{minipage}\ignorespacesafterend 
\end{equation} 
Suppose there are $t$ paths $P_1,\ldots,P_t$ as in \eqref{parity breaking paths}.
Without loss of generality, we may assume that the edge $\eta(v_iv_j)$ is directed towards $v_j$ for all $i,j\in[t]$, $i<j$.
Let $\{w_1,\ldots,w_t\}$ be the vertices of the following defined $K_t$-model $\eta'$.
Define $\eta'(w_i)=\eta(v_{2i-1})\cup \eta(v_{2i}) \cup E(P_i)$
and $\eta'(w_iw_j)=\eta(v_{2i}v_{2j-1})$ for $j>i$.

Let $i<j<k\leq t$ and let $C$ be the unique cycle in 
\begin{align*}
    \eta'(w_i) \cup \eta'(w_j) \cup \eta'(w_k) \cup \eta'(w_iw_j) \cup \eta'(w_iw_k) \cup \eta'(w_jw_k).
\end{align*}
Since $E(P_j) \subseteq E(C)$, $\g(P_j) \neq 0$, and all other edges of $C$ are null-labeled, we have $\g(C) \neq 0$.  Thus, $\eta'$ is a $\G$-odd $K_t$-model in $(\dirG, \G)$ which is an enlargement of $\eta$, as required.

\bigskip

Let us come back to the $\G$-bipartite $K_{24t}$-model $\eta$.
We now show that there is a collection of paths as described in (\ref{parity breaking paths})
or a set $X$ such that $|X|<8t$ and the $\zT_\pi$-large block of $(\dirG-X, \g)$ is $\G$-bipartite.
Recall that $V(\eta)=\{v_1,\ldots,v_{24t}\}$.
Pick from each tree $\eta(v_i)$ a vertex $s_i$ such that for all components $K$ in $\eta(v_i)-s_i$,
there exist at most $12t-1$ integers $j\in [24t]\setminus\{i\}$ for which an end of $\eta(v_iv_j)$ belongs to $K$ (clearly, $s_i$ exists).
Let $S=\{s_1,\ldots,s_{24t}\}$.

By Lemma~\ref{lemma: non-zero A-paths},
there is a set $X$ such that $|X|\leq 8t-2$ and $(\dirG-X, \g)$ does not contain a $\G$-non-zero $S$-path
or $(\dirG, \G)$ contains $4t$ disjoint $\G$-non-zero $S$-paths.
Suppose such a set $X$ exists.
Since $|X|<8t$, we may assume that $\eta(v_i)\cap X=\es$ for $i\in [16t]$.

Let $(\overrightarrow{U}, \g)$ be the $\cT_\eta$-large block of $(\dirG -X, \g)$.
Assume for a contradiction that $(\overrightarrow{U}, \g)$ is not $\G$-bipartite and hence contains a $\G$-non-zero cycle $C$.
For $i \in [2]$, if $s_i\notin V(U)$, then let $x_i$ be the unique vertex in $U\cap \eta(v_i)$ such that there is a $s_i$--$x_i$-path $Q_i$ internally-disjoint from $U$ in $\eta(v_i)$.
If $s_i\in V(U)$, set $x_i=s_i$.
Since $U$ is $2$-connected, there are two disjoint paths $Q_1',Q_2'$ joining $\{x_1,x_2\}$ and $C$.
Since $C$ is a $\G$-non-zero cycle, there is a $\G$-non-zero $s_1$--$s_2$-path in $(\dirG -X, \g)$ by combining $Q_1$, $Q_1'$, a suitable part of the cycle $C$,
$Q_2'$, and $Q_2$.
This contradicts the fact that $(\dirG -X, \g)$ does not contain a $\G$-non-zero $S$-path.
Hence, the $\zT_\pi$-large block of $(\dirG -X, \g)$ is $\G$-bipartite, as required.

Therefore, from now on, we assume that there are $4t$ disjoint $\G$-non-zero $S$-paths $\cP=\{P_1,\ldots, P_{4t}\}$.
We choose $\cP$ such that the number of edges lying in a path of $\cP$ but not in any tree $\eta(v_i)$ is as small as possible.
By re-indexing, we may assume $P_i$ has ends $s_{2i-1}$ and $s_{2i}$.
It is not difficult to see that $\bigcup_{i=8t+1}^{24t}\eta(v_i)$ does not intersect a path in $\cP$ by our choice of $\cP$.
 
Let $(\overrightarrow{G'}, \g')$ be the minor of $(\dirG, \g)$ obtained by contracting each $\eta(v_i)$ to a single vertex for $i\in [10t]\sm [8t]$
and let $S'=\{s_{8t+1}', \dots, s_{10t}'\}$ be the set of contracted vertices.

Again, by Lemma~\ref{lemma: non-zero A-paths},
there is a set $Y \subseteq V(G')$ such that $|Y|\leq 2t-2$ and $(\overrightarrow{G'}-Y, \g')$ does not contain a $\G$-non-zero $S'$-path
or $(\overrightarrow{G'}, \g')$ contains $t$ disjoint $\G$-non-zero $S'$-paths.

Suppose that there are $t$ disjoint $\G$-non-zero $S'$-paths  $R_1',\ldots,R_t'$ in $(\overrightarrow{G'}, \g')$.
It is easy to see that we may lift these paths to $t$ disjoint $\G$-non-zero $S$-paths $R_1,\ldots,R_t$ in $(\dirG, \g)$
such that for all $i \in [t]$, $R_i$ joins $s_{8t+2i-1}$ and $s_{8t+2i}$,
and $R_i$ is disjoint from $\eta(v_j)$ for all $j\in \{8t+1,\ldots,10t\}\setminus\{8t+2i-1,8t+2i\}$.
By applying (\ref{parity breaking paths}) to $\bigcup_{i=8t+1}^{10t}\eta(v_i) \cup R_i$, we obtain the desired $\G$-odd $K_t$-model.

Therefore, we may assume that there is a set $Y \subseteq V(G')$ such that $|Y|\leq 2t-2$ and $(\overrightarrow{G'}-Y, \g')$ does not contain a $\G$-non-zero $S'$-path.
Since $|Y|\leq 2t-2$ and $|\cP|=4t$, we may assume that $Y$ is disjoint from $\eta(v_1),\eta(v_2),P_1$ (possibly renaming among $[8t]$),
$s_{8t+1}', s_{8t+2}'$ (possibly renaming among $[10t]\sm[8t]$), and $\bigcup_{i=12t+1}^{24t}\eta(v_i)$ (possibly renaming among $[24t]\sm[10t]$).
Note that $P_1$ may only intersect $\eta(v_i)$ for $i \in \{3, \dots, 8t\}$. Let $a$ and $b$ be the ends of $\eta(v_1v_2)$ in $\eta(v_1)$ and $\eta(v_2)$, respectively.  
By our choice of $S$,
there exist distinct $j_1,j_2\in [24t]\sm[12t]$ such that $\eta(v_{j_2}v_1)$ either $a=s_1$ or $\eta(v_{j_2}v_1)$ does not meet the component of $\eta(v_1)-s_1$ containing $a$ and either $b=s_2$ or $\eta(v_2v_{j_1})$ does not meet the component of $\eta(v_2)-s_2$ containing $b$.  Therefore, the  unique (null-labeled) cycle $C$ in $\bigcup_{i\in \{1,2,j_1,j_2\}}\eta(v_i)\cup \eta(v_1v_2)\cup \eta(v_2v_{j_1})\cup \eta(v_{j_1}v_{j_2})\cup \eta(v_{j_2}v_1)$ contains $s_1$ and $s_2$.
Observe that $P_1\sm E(C)$ is a collection of (possibly trivial) $C$-paths.
As $C$ is null-labeled and $s_1,s_2\in V(C)$,
one of these $C$-paths is $\G$-non-zero because $P_1$ is $\G$-non-zero.
This gives rise to a $\G$-non-zero $s_{j_1}$-$s_{j_2}$-path that in fact contains a $\G$-non-zero $\eta(v_{j_1})$--$\eta(v_{j_2})$-path $P'$ (since $P_1$ is disjoint from $\eta(v_{j_1})\cup \eta(v_{j_2})$).
Consequently, $\eta(v_{j_1})\cup \eta(v_{j_2})\cup \eta(v_{j_1}v_{j_2}) \cup P'$ contains a $\G$-non-zero cycle $C'$.
This in turn, together with any $s_{8t+i}'$-$C'$-path with internal vertices in $\eta(v_{j_i})$ for $i\in [2]$ yields a $\G$-non-zero $s_{8t+1}'$--$s_{8t+2}'$-path in $(\overrightarrow{G'}, \g')$, which is a contradiction.
\end{proof}

We now extend Lemma~\ref{lemma: odd Kt minor} to $(\G_1,\G_2)$-group-labeled graphs.  

\begin{lemma}\label{lemma: odd Kt minor2}
For every $t \in \mathbb{N}$, there is an integer $T=T(t)$ such that for all  groups $\G_1$ and $\G_1$, 
every $K_T$-model $\pi$ in a $(\G_1 \oplus \G_2)$-labeled graph $(\dirG, \g)$ either contains a $(\G_1,\G_2)$-non-zero $K_t$-model that is an enlargement of $\pi$,
or a set $X \subseteq V(G)$ such that $|X|<8R(t,24t)$ and the $\zT_\pi$-large block of $(\dirG-X, \g)$ is $\G_1$-bipartite or $\G_2$-bipartite.
\end{lemma}
\begin{proof}
Let $r=R(t,24t)$, $T=R(r,24t)$, and $\pi$ be a $K_T$-model in a $(\G_1 \oplus \G_2)$-labeled graph $(\dirG, \g)$. By regarding $\pi$ as a $K_T$-model in $(\dirG, \g_1)$ and applying 
Lemma~\ref{lemma: odd Kt minor}, 
there exists a $\G_1$-odd $K_{r}$-model $\pi'$ in $(\dirG, \g_1)$,
or a set of less than $8r$ vertices $X_1$ such that the $\cT_\pi$-large block of $(\dirG-X_1, \g_1)$ is $\G_1$-bipartite.

We may assume the first case.   By regarding $\pi'$ as a $K_r$-model in $(\dirG, \g_2)$ and applying 
Lemma~\ref{lemma: odd Kt minor}, there exists a $\G_2$-odd $K_{t}$-model $\pi''$ in $(\dirG, \g_2)$ that is an enlargement of $\pi'$,
or a set of less than $8t$ vertices $X_2$ such that the $\cT_{\pi'}$-large block of $(\dirG-X_2, \g_2)$ is $\G_2$-bipartite.

Once again, we may assume that we have such a $\G_2$-odd $K_t$-model.
Note that $|\pi''(uv)|=1$ for every $uv$ of $\pi''$.
For every vertex $v$ of $\pi''$,
let $v'$ be a vertex  of $\pi'$ such that $\pi'(v')\subseteq \pi''(v)$. 
By adding $\pi'(u'v')$ to $\pi''(uv)$ for every edge $uv$ of $\pi''$,
we obtain a $(\G_1,\G_2)$-odd $K_t$-model in $(\dirG, \g)$.
\end{proof}

\section{A Flat Wall Theorem} \label{sec:flatwall}

In this section, we state our Flat Wall Theorem for $(\G_1 \oplus \G_2)$-labeled graphs.  

Let $\G_1$ and $\G_2$ be  groups. Let $(\dirW, \g)$ be a wall in a $(\G_1 \oplus \G_2)$-labeled graph $(\dirG, \g)$.  We say $(\dirW, \g)$ is \emph{facially $\G_i$-non-zero} if every facial cycle of $(\dirW, \g)$ is $\G_i$-non-zero.
If $(\dirW, \g)$ is both facially $\G_1$-non-zero and facially $\G_2$-non-zero, then we say that $(\dirW, \g)$ is \emph{facially $(\G_1, \G_2)$-non-zero}.

Further suppose that $(\dirW, \g)$ is a flat wall in  $(\dirG, \g)$.  Let $(\overrightarrow{W_0}, \g)$ be a flat $1$-contained subwall of $(\dirW, \g)$ and let $(A, B)$ be a separation certifying that $W_0$ is flat.
Let $N$ be the set of top nails of $W_0$ with respect to $W$.

Recall that  for a family $\cP$ of disjoint $N$-paths, $I_{\cP}\subseteq N$ is the minimal interval (under inclusion) containing all endpoints of paths in $\cP$, 
and $I_{\cP}^{\ell}$ and $I_{\cP}^{r}$ are the minimal intervals containing all the left and right endpoints of $\cP$, respectively.  In the following we tacitly assume that $N$-paths are always traversed from their left endpoint to their right endpoint.  

An $N$-linkage $\cP$ in $(\dirG, \g)$ is \emph{clean} with respect to $(A,B)$ if
\begin{enumerate}[(i)]
  \item the paths in $\cP$ are internally disjoint from $B$,
	\item $\cP$ is pure,
	\item every path in $\cP$ is $\G$-non-zero, and
	\item if $\cP$ is crossing or nested, then $\g(P_1)=\g(P_2)\not=0$ for all $P_1,P_2\in\cP$.
\end{enumerate}

We say that $\cP$ is \emph{$\G_i$-clean} if it is clean in $(\dirG, \g_i)$.  A pair of $N$-linkages $(\cP,\cQ)$ is \emph{$(\G_1,\G_2)$-clean} with respect to $(A,B)$ if 
\begin{enumerate}[(i)]
    \item $\cP$ and $\cQ$ are $\G_1$-clean and $\G_2$-clean, respectively,
	\item for all $P\in \cP,Q\in \cQ$, the paths $P$ and $Q$ are disjoint,
	\item $|\cP| = |\cQ|$, 
	\item if $I_\cP \cap I_\cQ = \emptyset$, then $I_\cP < I_\cQ$ and otherwise if $I_\cP \cap I_\cQ \neq \emptyset$ then $I^\ell_\cP < I^\ell_\cQ < I^r_\cP < I^r_\cQ$ and neither $\cP$ nor $\cQ$ is in series.
\end{enumerate}
The \emph{size} of $(\cP,\cQ)$ is $|\cP|$.
See Figure~\ref{fig: clean linkage} for illustration.
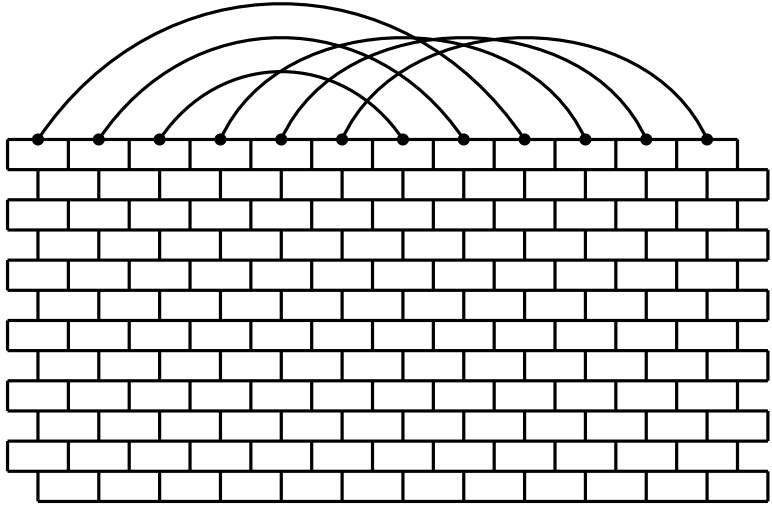
\begin{figure}[t]
\centering
\begin{tikzpicture}
\def\height{1}
\def\width{1}
\def\sizeofwalla{12}
\def\sizeofwallb{12}

\begin{scope}[scale=0.4]
\foreach \i in {2,...,\sizeofwallb}{
	\draw[very thick] (0,\i-1)--(2*\sizeofwalla+1,\i-1);
}
\draw[very thick] (1,0)--(2*\sizeofwalla+1,0);
\draw[very thick] (0,\sizeofwallb)--(2*\sizeofwalla,\sizeofwallb);

\pgfmathsetmacro{\bh}{\sizeofwallb/2-1}
\foreach \i in {0,1,...,\sizeofwalla}{
	\foreach \j in {0,1,...,\bh}{
		\draw[very thick] (2*\i+1,2*\j)--(2*\i+1,2*\j+1);
	}
}

\foreach \i in {0,1,...,\sizeofwalla}{
	\foreach \j in {0,1,...,\bh}{
		\draw[very thick] (2*\i,2*\j+1)--(2*\i,2*\j+2);
	}
}	

\begin{scope}[shift={(-1,0)}]
\draw[very thick] (2,12)..controls (6,18) and (14,18) .. (18,12);
\draw[very thick] (4,12)..controls (7,16.5) and (13,16.5) .. (16,12);
\draw[very thick] (6,12)..controls (8,15) and (12,15) .. (14,12);

\draw[very thick] (8,12)..controls (10,16.5) and (18,16.5) .. (20,12);
\draw[very thick] (10,12)..controls (12,16.5) and (20,16.5) .. (22,12);
\draw[very thick] (12,12)..controls (14,16.5) and (22,16.5) .. (24,12);

\foreach \i in {1,2,...,12}{
	\fill[very thick] (2*\i,12) circle (0.2);
}
\end{scope}

\end{scope}

\end{tikzpicture}
\caption{An example of a $(\G_1 , \G_2)$-clean pair of linkages of size 3.
The first three paths have the same (non-zero) group value in the first coordinate, and the other three paths have the same (non-zero) group value in the second coordinate.  
}\label{fig: clean linkage}
\end{figure}
We can now state our refinement of the Flat Wall Theorem.

\begin{theorem} \label{thm: main wall}
For every $t\in \N$, there exist integers $T(t)$ and $g(t)$
with the following property. 
Let $\G_1$ and $\G_2$ be  groups and $(\dirG, \g)$ be a $(\G_1 \oplus \G_2)$-labeled graph.
If $(\dirG, \g)$ contains a $T(t)$-wall $(\dirW, \g)$,
then one of the following statements holds.
\begin{enumerate}[(a)]
	\item There is a $(\G_1,\G_2)$-odd $K_t$-model $\pi$ in $(\dirG, \g)$ such that $\cT_\pi$ is a restriction of $\cT_W$.
	\item There is a set of vertices $Z$
	such that $|Z|\leq g(t)$ and there is a flat $100t$-wall $(\overrightarrow{W_0}, \g)$ in $(\dirG-Z, \g)$ with top nails $N_0$ and certifying separation	$(A_0, B_0)$ such that
	$\cT_{W_0}$ is a restriction of $\cT_W$ and after possibly shifting
	\begin{enumerate}
		\item[(b.i)] $(\overrightarrow{W_0}, \g)$ is facially $(\G_1,\G_2)$-non-zero, or
		\item[(b.ii)] for some $i\in [2]$, $(\overrightarrow{B_0}, \g_i)$ is null-labeled, $(\overrightarrow{W_0}, \g)$ is facially $\G_{3-i}$-non-zero,
		and there is a $\G_i$-clean $N_0$-linkage $\cP$ of size $t$ 
		with respect to $(A_0,B_0)$, or
		\item[(b.iii)] $(\overrightarrow{B_0}, \g)$ is null-labeled
		and there is a $(\G_1,\G_2)$-clean pair $(\cP,\cQ)$ of $N_0$-linkages of size $t$ with respect to $(A_0,B_0)$.
  \end{enumerate}
  \item There is a set $Z$ of vertices of $(\dirG, \g)$ and some $i\in [2]$ such that $|Z|\leq g(t)$ and the $\cT_W$-large block of $$(\dirG-Z, \g)$$ is $\G_i$-bipartite.
\end{enumerate}
\end{theorem}

The Flat Wall Theorem for a group-labeled graph $(\dirG,\g')$ with only one group $\G$ follows from Theorem~\ref{thm: main wall} by considering the labeling $\g(e)=(\g'(e),\g'(e))$ for all edges $e\in E(G)$.
After establishing some necessary lemmas in the next two sections, we will prove Theorem \ref{thm: main wall} in Section \ref{sec:thmmain}.

\section{Cleaning Paths and Cycles} \label{sec:cleaning}

In this section, we prove some lemmas towards the proof of Theorem \ref{thm: main wall}.  Most of these lemmas are of the following form: given a family of paths or cycles, 
there is also a subfamily with certain nice properties.  

\begin{lemma}\label{lemma: two odd triangles}
Let $\G_1$ and $\G_2$ be  groups and $(\dirG, \g)$ be a $(\G_1 \oplus \G_2)$-labeled graph that
contains two disjoint cycles $C_1$ and $C_2$ and two disjoint $C_1$--$C_2$-paths
$P_1$ and $P_2$ such that $C_1$ is $\G_1$-non-zero and $C_2$ is $\G_2$-non-zero.  
Then $C_1 \cup C_2 \cup P_1 \cup P_2$ contains a $(\G_1, \G_2)$-non-zero cycle. 
\end{lemma}
\begin{proof}
We may assume that $\g_2(C_1)=0$ and $\g_1(C_2)=0$; otherwise $C_1$ or $C_2$ is the desired $(\G_1, \G_2)$-non-zero cycle.  
Let $C'$ be a cycle in $C_1 \cup C_2 \cup P_1 \cup P_2$
which contains $P_1 \cup P_2$.   
Again, if $C'$ is $(\G_1, \G_2)$-non-zero, then we are done.  
Suppose $v$ is the first vertex of $C'$.  Let $C_i'$ be the cycle obtained by starting at $v$ and following $C'$ until intersecting $V(C_i)$, then following $C_i\setminus C'$ until intersecting $V(C')$, then following $C'$ to $v$. Observe that $\g_{3-i}(C_i')=\g_{3-i}(C')$ 
and $\g_{i}(C_i')\neq\g_{i}(C')$.
Hence if $\g_i(C') =0$ and $\g_{3-i}(C') \neq 0$, then $C_i'$ is
a $(\G_1, \G_2)$-non-zero cycle. 
Finally, if $\g(C')=(0,0)$, then $((C_1 \cup C_2) - C') \cup P_1 \cup P_2$ is a $(\G_1, \G_2)$-non-zero cycle.
\end{proof}

\begin{lemma}\label{lemma: non-zero brick}
Let $\G_1$ and $\G_2$ be  groups, $(\dirG, \g)$ be a $(\G_1 \oplus \G_2)$-labeled graph, and $C,C_1,C_2$ be disjoint cycles of $(\dirG, \g)$.  
Let $P_1,P_1',P_2,P_2'$ be pairwise disjoint paths such that $P_i$ and $P_i'$ are $C$--$C_i$-paths with $V(P_i) \cap V(C_{3-i})=\emptyset=V(P_i') \cap V(C_{3-i})$ for $i\in[2]$.
Let $p_1,p_1',p_2,p_2'$ be the ends of $P_1,P_1',P_2,P_2'$ on $C$ and assume that they occur in this cyclic order on $C$.  
Further assume that $C_i$ is $\G_i$-non-zero for every $i\in [2]$ and $C_1$ is $\G_2$-zero.  Let $I_1$ be the subpath of $C$ with ends $p_1$ and $p_2'$ and with $p_2 \notin V(I_1)$.  
Let $I_2$ be the subpath of $C$ with ends $p_1'$ and $p_2$ and with $p_1 \notin V(I_2)$.  
Then there is a $(\G_1,\G_2)$-non-zero cycle $D$ in $C_1 \cup C_2 \cup C \cup P_1 \cup P_1'\cup P_2 \cup P_2' $ containing $I_1 \cup I_2$.   
\end{lemma}
\begin{proof}
We use the notation as in the statement.
Moreover,
let the endpoints of $P_i$ and $P_i'$ on $C_i$ be denoted by $q_i$ and $q_i'$, respectively.
The vertices $q_i$ and $q_i'$ split $C_i$ into $q_i$--$q_i'$-paths $Q_i^a$ and $Q_i^b$.
By shifting, we may assume that all edges in $I_1 \cup I_2 \cup P_1 \cup P_1' \cup P_2 \cup P_2'$ are null-labeled.

Suppose $\g(Q_i^a)=(x_i^a,y_i^a)$ and $\g(Q_i^b)=(x_i^b, y_i^b)$.  Since $C_1$ is $\G_2$-zero, we have $y_1^a=y_1^b:=y_1$.  Since $C_2$ is $\G_2$-non-zero, by symmetry, we may assume that $y_1 + y_2^a \neq 0$.  Since $C_1$ is $\G_1$-non-zero, by symmetry, we may assume that $x_1^a+x_2^a \neq 0$.  We are now done by letting $D$ be 
$$
I_1 \cup I_2 \cup P_1 \cup P_1' \cup P_2 \cup P_2' \cup Q_1^a \cup Q_2^a. \qedhere
$$
\end{proof}

\begin{lemma}\label{lemma: chords one type}
Let $t\in\N$, $G$ be a graph, $L \subseteq V(G)$ be linearly ordered, and $\cP$ be an $L$-linkage.  
If $|\cP| \geq t^3$, then
$\cP$ contains a pure $L$-linkage $\cP'$ of size $t$.
\end{lemma}

\begin{proof}
Dilworth's theorem \cite{Dilworth50} implies that 
for every pair of positive integers $r,s$,
every collection of $rs$ intervals contains $r$ disjoint intervals or $s$ pairwise intersecting intervals.
For every path $P\in\cP$, 
we associate an interval, namely, the smallest interval containing both endpoints of $P$.
By Dilworth's theorem,
we have $t$ disjoint intervals or $t^2$ pairwise intersecting intervals.
In the first case, $\cP$ contains $t$ paths in series.
So assume there are $t^2$ paths $\cQ$ such that their intervals contain a common point.
We order these paths by their smaller endpoint.
Their larger endpoints induce a permutation of their smaller endpoints.
By the Erd\H{o}s-Szekeres theorem~\cite{ES35}, there is a monotone subsequence of order $t$.
If this subsequence is increasing, then $\cP$ contains a crossing subfamily of size $t$, and if it is decreasing, then $\cP$ contains a nested subfamily of size $t$.
\end{proof}

\begin{lemma}\label{lemma: chords separating}
Let $t\in\N$, $G$ be a graph, $L \subseteq V(G)$ be a linearly ordered set, and $\cP,\cQ$ be two pure $L$-linkages of size $4t$ such that
$\cP \cup \cQ$ is a $L$-linkage of size $8t$.
Then, there exist two pure linkages $\cP'\subseteq \cP,\cQ'\subseteq \cQ$ of size $t$ such that
\begin{itemize}
    \item  $I_{\cP'}$ and $I_{\cQ'}$ are disjoint if $\cP, \cQ$ are in series,
    \item  $I_{\cP'}$ is disjoint from $I_{\cQ'}^{\ell} \cup I_{\cQ'}^{r}$ if $\cP$ is in series and $\cQ$ is not in series,
    \item  $I_{\cQ'}$ is disjoint from $I_{\cP'}^{\ell} \cup I_{\cP'}^{r}$ if $\cQ$ is in series and $\cP$ is not in series, and
    \item $I_{\cP'}^{\ell},I_{\cP'}^{r},I_{\cQ'}^{\ell},I_{\cQ'}^{r}$ are pairwise disjoint if neither $\cP,\cQ$ is in series.
\end{itemize}
\end{lemma}
\begin{proof}
We order the paths of $\cP$ and $\cQ$ as $\{P_1, \dots, P_{4t}\}$ and $\{Q_1, \dots, Q_{4t}\}$ according to the order of their left endpoints.  
Let $p_{i}^{\ell}$ and $p_{i}^{r}$ be the left and right endpoint of $P_i$ and $q_{i}^{\ell}$ and $q_{i}^{r}$ be the left and right endpoint of $Q_i$, respectively.

First suppose that $\cP$ is in series. If $q_{3t}^{\ell}$ is to the right of $p_{t}^{r}$, 
then $\cP'=\{P_1,\ldots,P_t\}$ and $\cQ'=\{Q_{3t+1},\ldots,Q_{4t}\}$ are two families of paths with the desired properties.
So we may assume that $q_{3t}^{\ell}$ is to the left to $p_{t}^{r}$.  If $\cQ$ is also in series, then we may take $\cP'=\{P_{t+1},\ldots,P_{2t}\}$ and $\cQ'=\{Q_1,\ldots,Q_{t}\}$.  Thus $\cQ$ is either nested or crossing.  If there are at least $3t$ right endpoints of $\cQ$ to the right of $p_{2t}^{r}$, then we may take $\cP'=\{P_{t+1},\ldots,P_{2t}\}$ and $\cQ'=\{Q_{t+1},\ldots,Q_{2t}\}$. 
Otherwise, we take $\cQ'$ to be a family of $t$ paths of $\cQ$ all of whose right endpoints are to the left of $p_{2t}^{r}$, and $\cP'=\{P_{2t+1},\ldots,P_{3t}\}$.

For the rest of the proof, we may assume that neither $\cP$ nor $\cQ$ are in series.
For $i \in [4]$ define $X_i^\ell:=[p_{1+(i-1)t}^{\ell}, p_{it}^{\ell}]$ and $X_i^r:=[p_{1+(i-1)t}^{r}, p_{it}^{r}]$. Similarly, define $Y_i^\ell:=[q_{1+(i-1)t}^{\ell}, q_{it}^{\ell}]$ and $Y_i^r:=[ q_{1+(i-1)t}^{r}, q_{it}^{r}]$.  It suffices to find indices $a$ and $b$ such that $X_a^\ell \cup X_a^r$ is disjoint from $Y_b^\ell \cup Y_b^r$.  

Let $H$ be the graph with vertex set $\{X_i^\ell,X_i^r,Y_i^\ell,Y_i^r\}_{i\in [4]}$
and two vertices are adjacent if they intersect in $L$ (considered as two intervals).
By construction, $H$ is an interval graph and bipartite.
It is well known that interval graphs are chordal (that is, they do not contain induced cycles of length at least $4$).
Thus, $H$ is a forest, because it is triangle-free.
As $H$ has 16 vertices, it has at most $15$ edges.

If $X_a^\ell \cup X_a^r$ is not disjoint from $Y_b^\ell \cup Y_b^r$ for some $a,b\in[4]$,
then $H_{a,b}=H[\{X_a^\ell, X_a^r,Y_b^\ell, Y_b^r\}]$ contains an edge.
As $H_{a,b}$ is edge-disjoint from $H_{a',b'}$ for $(a,b)\neq (a',b')$
and there are 16 tuples $(a,b)$ with $a,b\in[4]$, but only 15 edges in $H$,
there exist $(a,b)$ such that $H_{a,b}$ does not contain an edge.
Therefore, $X_a^\ell \cup X_a^r$ is disjoint from $Y_b^\ell \cup Y_b^r$ as desired.
\end{proof}

\begin{lemma}\label{lemma: chords making them disjoint}
Let $t\in \N$, let $\G_1$ and $\G_2$ be  groups and $(\dirG, \g)
$ be a $(\G_1 \oplus \G_2)$-labeled graph.
Let $S \subseteq V(G)$.
If $(\dirG, \g)$ contains $3t$ disjoint $\G_1$-non-zero $S$-paths and $t$ disjoint $\G_2$-non-zero $S$-paths, then $(\dirG, \g)$ contains $2t$ disjoint $S$-paths $P_1,\ldots,P_{2t}$ such that $P_i$ is $\G_1$-non-zero for $i\leq t$ and $\G_2$-non-zero for $ t<i\leq 2t$.
\end{lemma}
\begin{proof}
Let $\cQ=\{Q_1,\ldots,Q_{3t}\}$ be disjoint $\G_1$-non-zero $S$-paths and
let $\mathcal{R}=\{R_1,\ldots,R_t\}$ be disjoint $\G_2$-non-zero $S$-paths such that the number of edges belonging to a path in $\mathcal{R}$ but not to a path $\mathcal{Q}$ is as small as possible.

If the paths in $\mathcal{R}$ intersect at most $2t$ paths of $\mathcal{Q}$, then we are done.  
Thus there is a path in $\mathcal{Q}$, say $Q_1$, that intersects a path in $\mathcal{R}$, and the endpoints of $Q_1$ do not belong to a path in $\mathcal{R}$.
Let $q$ be an endpoint of $Q_1$ and let $r$ be the first vertex of $Q_1$ after $q$ contained in a path in $\mathcal{R}$, say $R_1$. Note that the subpaths $R_1r$ and $rR_1$ of $R_1$
both contain at least one edge that is not contained in a path in $\cQ$.
Moreover, since $R_1$ is a $\G_2$-non-zero path, at least one of the paths $qQ_1r \cup R_1r$ or $qQ_1r \cup rR_1$ is $\G_2$-non-zero, say $qQ_1r \cup R_1r$.  Replacing $R_1$ by $qQ_1r \cup R_1r$ in $\mathcal{R}$ contradicts the choice of $\mathcal{R}$. 
\end{proof}

\begin{lemma}\label{lemma: chords complete}
Let $\G_1$ and $\G_2$ be  groups and $(\dirG, \g)$ be a $(\G_1 \oplus \G_2)$-labeled graph.
Let $L$ be a linearly ordered subset of vertices of $G$ and $t\in\N$.
If $(\dirG, \g)$ contains a $\G_1$-non-zero $L$-linkage of size $192t^3$ and a $\G_2$-non-zero $L$-linkage of size $64t^3$, 
then $(\dirG, \g)$ contains two $L$-linkages $\cP$ and $\cQ$ of size $t$ such that 
\begin{enumerate}[(i)]
\item $\cP\cup\cQ$ is an $L$-linkage of size $2t$,
\item the paths in $\cP$ and $\cQ$ are $\G_1$-non-zero and $\G_2$-non-zero, respectively,
\item $\cP,\cQ$ are both pure, and
\item \begin{itemize}
    \item  $I_{\cP}$ and $I_{\cQ}$ are disjoint if $\cP, \cQ$ are both in series,
    \item  $I_{\cP}$ is disjoint from $I_{\cQ}^{\ell},I_{\cQ}^{r}$ if $\cP$ is in series and $\cQ$ is not in series,
    \item  $I_{\cQ}$ is disjoint from $I_{\cP}^{\ell},I_{\cP}^{r}$ if $\cQ$ is in series and $\cP$ is not in series, and
    \item $I_{\cP}^{\ell},I_{\cP}^{r},I_{\cQ}^{\ell},I_{\cQ}^{r}$ are pairwise disjoint if neither $\cP$ nor $\cQ$ are in series.
\end{itemize}  
\end{enumerate}
\end{lemma}
\begin{proof}
First use Lemma~\ref{lemma: chords making them disjoint} to obtain a $\G_1$-non-zero $L$-linkage $\cP''$ of size $64t^3$ and a $\G_2$-non-zero $L$-linkage $\cQ''$ of size $64t^3$ such that $\cP''\cup\cQ''$ is an $L$-linkage of size $128t^3$.
By Lemma~\ref{lemma: chords one type},
there are pure $L$-linkages $\cP'\subseteq \cP'',\cQ'\subseteq \cQ''$ each of size $4t$.
Finally, Lemma~\ref{lemma: chords separating} applied to $\cP',\cQ'$ completes the proof.
\end{proof}

We now show how two large non-zero linkages attaching to a wall can be used to find a $(\G_1,\G_2)$-clean pair of linkages for a subwall.  We first need an easy proposition on flat walls; we omit the proof.

\begin{proposition}\label{lemma: boundary cycle excluded}
Let $G$ be a graph and $r \in \N \sm \{1\}$.  
Let $W$ be a flat $r$-wall in $G$ and $(A, B)$ be a certifying separation for $W$.  
Let $X$ be the four corners of $W$ and let $C$ be the boundary cycle of $W$.  
Let $H$ be a connected subgraph of $G$ which is disjoint from $W$ and such that $V(H)$ has a neighbor in each of the four components of $C-X$.  
Then $H \subseteq A$.  Specifically, if $W$ is a $1$-contained subwall of a larger wall $W'$ with boundary cycle $C'$, then $C' \subseteq A$.
\end{proposition}

\begin{lemma}\label{lemma: wall paths}
Let $\G_1$ and $\G_2$ be  groups and $(\dirG, \g)$ be a $(\G_1 \oplus \G_2)$-labeled graph.
Let $t\in \N\sm\{1\}$ and $W$ be a flat wall in $G$.
Suppose $W_1$ is a $1$-contained flat $10^6t^6$-subwall of $W$
and let $W_2$ be a $10t$-contained flat $10^2t$-subwall of $W_1$. For $i \in [2]$, let $(A_i, B_i)$ be a certifying separtion for $W_i$.  Let $N_i$ be the set of top nails of $W_i$ with respect to $W$.  
If $(\overrightarrow{B_1}, \g)$ is null-labeled and there exist  $\G_i$-non-zero $N_1$-linkages of size $10^5t^6$ in $(\overrightarrow{A_1} - (V(B_1) \setminus N_1), \g)$ for both $i\in [2]$,
then there exists a $(\G_1,\G_2)$-clean pair $(\cP,\cQ)$ of $N_2$-linkages for $\dirW_2$ of size $t$.
\end{lemma}

\begin{proof}
First, we apply Lemma~\ref{lemma: chords complete} in $(\overrightarrow{A_1} - (V(B_1) \setminus N_1), \g)$ to obtain two $N_1$-linkages $\cP_1$ and $\cQ_1$ each of size $t^2$ such that
\begin{enumerate}[(i)]
\item $\cP_1\cup\cQ_1$ is a linkage of size $2t^2$,
\item the paths in $\cP_1$ and $\cQ_1$ are $\G_1$-non-zero and $\G_2$-non-zero, respectively,
\item $\cP_1,\cQ_1$ are pure, and
\item \begin{itemize}
    \item  $I_{\cP_1}$ and $I_{\cQ_1}$ are disjoint if $\cP_1, \cQ_1$ are in series,
    \item  $I_{\cP_1}$ is disjoint from $I_{\cQ_1}^{\ell},I_{\cQ_1}^{r}$ if $\cP_1$ is in series and $\cQ_1$ is not in series,
    \item  $I_{\cQ_1}$ is disjoint from $I_{\cP_1}^{\ell},I_{\cP_1}^{r}$ if $\cQ_1$ is in series and $\cP_1$ is not in series, and
    \item $I_{\cP_1}^{\ell},I_{\cP_1}^{r},I_{\cQ_1}^{\ell},I_{\cQ_1}^{r}$ are pairwise disjoint if neither $\cP_1$ nor $\cQ_1$ are in series.
\end{itemize}
\end{enumerate}
We regard each path in $\cP_1 \cup \cQ_1$ as being traversed from its left endpoint to its right endpoint.  
Next, we define two new families $\cP_2$ and $\cQ_2$ of paths.
If $\cP_1$ is in series, then let $\cP_2$ be an arbitrary subset of $t$ paths in $\cP_1$.
If $\cP_1$ is nested or crossing and contains $t$ paths with the same group value,
then let $\cP_2$ be a set of such paths.
In the other cases, we order the paths in $\cP$ as $P_1,\ldots,P_{t^2}$  according to their left endpoints.
For $i\in [t]$, choose $a_i > (i-1)t+1$ such that $\g_1(P_{a_i})\not=\g_1(P_{(i-1)t+1})$ with $a_i$ minimum.
Note that this implies that $a_i\leq it$.  Let $H$ be the topmost horizontal path of $W_1$.  
For each $i \in [t]$, combine $P_{(i-1)t+1}$, $P_{a_i}$ and the subpath of $H$ between the right endpoints of $P_{(i-1)t+1}$ and $P_{a_i}$ to a path
and let $\cP_2$ be the collection of these paths (observe that in this case the paths in $\cP_2$ are in series). 
We do the same accordingly for $\cQ_1$ to obtain $\cQ_2$. 
Note that $\cP_2\cup \cQ_2$ is a collection of $2t$ disjoint paths and all paths in $\cP_2$ are $\G_1$-non-zero and all paths in $\cQ_2$ are $\G_2$-non-zero.

Since $(\overrightarrow{B_1}, \g)$ is null-labeled, it is easy to see that $\cP_2\cup \cQ_2$ can be extended to a $(\G_1,\G_2)$-clean pair of linkages for $(\dirW_2, \g)$ by using paths in $W_1 - W_2$ and the fact that this subgraph is contained in $A_2$ by Proposition \ref{lemma: boundary cycle excluded}. 
\end{proof}

The following lemma is a simplified version of Lemma~\ref{lemma: wall paths}.
The proof is the same as for Lemma~\ref{lemma: wall paths}.
\begin{lemma}\label{lemma: wall paths1}
Let $\G$ be a group and $(\dirG, \g)$ be a $\G$-labeled graph.
Let $t\in \N\sm\{1\}$ and $W$ be a flat wall in $G$.
Suppose $W_1$ is a $1$-contained flat $10^6t^6$-subwall of $W$
and let $W_2$ be a $10t$-contained flat $10^2t$-subwall of $W_1$. 
For $i \in [2]$, let $(A_i, B_i)$ be a certifying separtion for $W_i$.  Let $N_i$ be the set of top nails of $W_i$ with respect to $W$.  
If $(\overrightarrow{B_1}, \g)$ is null-labeled and there exist $\G$-non-zero $N_1$-linkages of size $10^5t^6$ in $(\overrightarrow{A_1} - (V(B_1) \sm N_1), \g)$,
then there exists a clean $N_2$-linkage $\cP$ for $\dirW_2$ of size $t$.
\end{lemma}

\section{Wall Lemmas} \label{sec:walllemmas}

In this section, we present several results concerning walls which we will use in the proof 
of Theorem \ref{thm: main wall}.
Recall, for a $t$-wall $W$, there exists a tangle $\cT$ of order $t+1$ associated with $W$.
Explicitly, $\cT$ consists of all $t$-separations $(A,B)$ of $W$ such that $B$ contains an entire horizontal path (or equivalently an entire vertical path) of $W$.
If $W$ is a minor of $G$ (possibly $G=W$), then we let $\cT_W$ denote the tangle in $G$ induced by $\cT$.

\begin{theorem}[\cite{RS91} (7.5)]\label{lemma: tangle wall}
For every positive integer $t$, there is an integer $T(t)$ with the following property.
For every graph $G$ and every tangle $\cT$ of order $T(t)$ in $G$, 
there is a $t$-wall $W$ in $G$ such that $\cT_W$ is a truncation of $\cT$.
\end{theorem}

\begin{lemma}\label{lemma: wall tangle}
Let $s,t$ be positive integers and $G$ be a graph.
If $W$ is a $t$-wall in $G$ and $W'$ is an $s$-subwall of $W$, then 
the tangle $\cT_{W'}$ is a restriction of $\cT_W$.
\end{lemma}
\begin{proof}
Let $(A, B)$ be an $s$-separation in $G$ and assume that $(A,B)\in \cT_{W'}$.   Let $P_0, \dots, P_{s}$ be the 
horizontal paths of $W$ such that each contains as a subpath a horizontal path of $W'$.
Since, $|V(A) \cap V(B)| \leq s$, there is some $P_j$ such that $V(P_j)$ is disjoint from $V(A) \cap V(B)$.
By the definition of $\cT_{W'}$ it follows that $P_j \subseteq B$, and so $(A, B) \in \cT_W$, as required. 
\end{proof}

Let $W$ be an $r$-wall and let $X$ be the branch vertices of the subdivision of the elementary $r$-wall.  That is, $X$ is the set of vertices of $W$ corresponding to the vertices of the elementary $r$-wall before subdividing edges.  Let $x,y \in X$ be contained in a common vertical path $P$ and common brick $B$ such that the subpath $P'$ of $P$ with ends $x$ and $y$ does not have any internal vertex in $X$ and $P'$ is not contained in the boundary cycle of $W$.  
Let $Q$ be a $W$-path with ends $x$ and $y$ and let $W'$ be the wall obtained from $W$ by deleting the edges and internal vertices of $P'$ and adding the path $Q$.  In effect, we reroute the vertical path $P$ through the path $Q$ as opposed to the subpath $P'$. 
We say that $W'$ is obtained from $W$ via a \emph{local rerouting}.  We require the following easy proposition.
We leave the proof to the reader.

\begin{proposition} \label{prop: rerouting}
Let $W$ be a flat wall in a graph $G$. If $W'$ is a local rerouting of $W$, then $W'$ is also a flat wall in $G$ and $\cT_W=\cT_{W'}$.
\end{proposition}

We continue with two lemmas about walls which are the counterparts 
to Lemma~\ref{lemma: odd Kt minor} and \ref{lemma: odd Kt minor2} about $K_t$-minors.

\begin{lemma}\label{lemma: one group wall}
Let $t\in \N$ and $t\ge 3$, let $\G$ be a group and $(\dirG, \g)$ be a $\G$-labeled graph.
Let $(\dirW, \g)$ be a flat $3t^2$-wall in $(\dirG, \g)$.  Then there exists a flat $t$-wall $(\dirWone, \g)$ with certifying separation $(A_1, B_1)$ such that either
\begin{enumerate}[(i)]
\item the block of $(\overrightarrow{B_1}, \g)$ containing  $(\dirWone, \g)$ is $\G$-bipartite, or
\item $(\dirWone, \g)$ is facially $\G$-non-zero.
\end{enumerate}
Moreover, $\zT_{W_1}$ is a restriction of $\zT_W$ and the boundary cycle of $W$ is contained in $A_1$.  
\end{lemma}
\begin{proof}
Let $(\dirW, \g)$ be a flat $3t^2$-wall in $(\dirG, \g)$.  
The collection of all $((2t-1)i + 1)$-th horizontal paths and all $(2tj+1)$-th vertical paths for $i,j\in \{0, \dots, t+1\}$
induces a 1-contained flat $(t+1)$-subwall $(\dirW_2, \g)$ of $(\dirW, \g)$.  
Note the interior of each brick $D_{i,j}$ of $(\dirW_2, \g)$
 contains a $t$-subwall $(\overrightarrow{W_{i,j}}, \g)$ of $(\dirW, \g)$ for all $i, j \in [t+1]$.  As a subwall of $W$, the wall $W_{i, j}$ is flat.  Fix $(A_{i, j}, B_{i, j})$ to be a certifying separation for $W_{i, j}$ which minimizes $|V(B_{i, j})|$.  This ensures that $B_{i, j}$ is connected and therefore $B_{i, j}$ is disjoint from $B_{i', j'}$ if either $i \neq i'$ or $j \neq j'$.   
 
Assume that there exists $i,j\in\ [t+1]$ such that the block of $(\overrightarrow{B_{i, j}}, \g)$ containing $(\overrightarrow{W_{i,j}}, \g)$  is $\G$-bipartite.  As $W_{i, j}$ is a subwall of $W$, it follows that $\zT_{W_{i, j}}$ is a restriction of $\zT_W$ by Lemma \ref{lemma: wall tangle} and the boundary cycle of $W$ is contained in $A_{i, j}$ by Proposition \ref{lemma: boundary cycle excluded}.  Thus, the theorem holds.

We may therefore assume that for all $i, j \in [t+1]$, the block of $(\overrightarrow{B_{i, j}}, \g)$ containing $(\overrightarrow{W_{i,j}}, \g)$ also contains a $\G$-non-zero cycle $C_{i,j}$.  Let $H_{i, j}$ be the union of subpaths of the horizontal paths of $W$ that are $W_{2}$-paths which meet $W_{i, j}$. Let $i, j \in [t]$. 
Let $R_{i,j}$ be the subpath of the right vertical path of $D_{i,j}$ with endpoints in distinct horizontal paths of $W_2$.  By our choice of $i$ and $j$, the path $R_{i, j}$ is not contained in the boundary cycle of $W_2$.  Observe that there are two disjoint $R_{i,j}$--$C_{i,j}$ paths contained
in $B_{i,j} \cup H_{i, j}$.  Therefore, since $C_{i,j}$ is $\G$-non-zero, there is a local rerouting of $(\dirW_2, \g)$ along $R_{i,j}$ such that the $(i, j)$-th brick is $\G$-non-zero.  Note as well by the construction and Lemma~\ref{lemma: boundary cycle excluded} that after the local rerouting, the resulting wall is disjoint from the boundary cycle of $W$.

We sequentially perform these local reroutings along the right vertical path $R_{i,j}$ of each brick of $(\dirW_2, \g)$ in lexicographic order $(1,1), (2,1), \dots, (t,t)$.  By Lemma~\ref{lemma: wall tangle} and Proposition~\ref{prop: rerouting}, the resulting $(t+1)$-wall is flat and the induced tangle is a restriction of $\cT_W$.  Moreover, all the facial cycles except the last vertical and horizontal row are $\G$-non-zero.  
To complete the proof, we delete the $(t+2)$-th horizontal and $(t+2)$-th vertical path to get a wall $(\dirWone, \g)$ which is facially non-zero and flat, satisfying (ii).
Fix a certifying separation $(A_1, B_1)$ of $(\dirWone, \g)$.  Let $D_1$ be the boundary cycle of $W_1$.  We conclude that the boundary cycle $D$ of $W$ is contained in $A_1$ by Proposition~\ref{lemma: boundary cycle excluded} applied to the component of $W-D_1$ containing $D$.  This completes the proof.
\end{proof}

We now extend Lemma \ref{lemma: one group wall} to the case when the graph is $(\G_1,\G_2)$-group-labeled.  
The proof follows the proof of Lemma \ref{lemma: one group wall}, but the construction is slightly more complicated as we must use Lemma \ref{lemma: non-zero brick} to reroute the bricks to be $(\G_1,\G_2)$-non-zero.  

\begin{lemma}\label{lemma: two group wall}
Let $\G_1$ and $\G_2$ be  groups, $(\dirG, \g)$ be a $(\G_1 \oplus \G_2)$-labeled graph,  
and $t\in \N$ with $t\ge 10$.  
Let $(\dirW, \g)$ be a flat $9t^3$-wall in $(\dirG, \g)$.  
Then there exists a flat $t$-wall $(\dirW_1, \g)$ with certifying separation $(A_1, B_1)$ that arises from a $t$-subwall of $W$ via local rerouting such that one of the following statements holds.
\begin{enumerate}[(i)]
	\item $(\dirW_1, \g)$ is facially $(\G_1,\G_2)$-non-zero,
	\item for some $i\in [2]$, the wall $(\dirW_1, \g)$ is facially $\G_i$-non-zero and the block of $(\overrightarrow{B_1}, \g)$ containing $(\dirW_1, \g)$  is $\G_{3-i}$-bipartite, or 
	\item the block of $(\overrightarrow{B_1}, \g)$ containing $(\dirW_1, \g)$ is $(\G_1 \oplus \G_2)$-bipartite.
\end{enumerate}
Moreover, $\cT_{W_1}$ is a restriction of $\cT_W$.
\end{lemma}
\begin{proof}
Let $(\dirW, \g)$ be a flat $9t^3$-wall in $(\dirG, \g)$.
The collection of all $((8t^2-1)i+2)$-th horizontal paths and $(8t^2j+1)$-th vertical paths of $(\dirW, \g)$  for all $i,j\in\{0,\ldots,t+1\}$ induces a flat $(t+1)$-subwall $(\dirW_2, \g)$ of $(\dirW, \g)$.

For all $i,j \in [t+1]$, the interior of every brick $D_{i,j}$ of $(\dirW_2, \g)$ contains two disjoint flat $3t^2$-subwalls $(\dirW_{3,i,j}, \g)$ and $(\dirW_{4,i,j}, \g)$ such that they use the same set of vertical paths of $W$, and $W_{4,i,j}$ is above $W_{3,i,j}$.
For $k \in \{3,4\}$, fix a certifying separation $(A_{k, i, j}, B_{k, i, j})$ of $W_{k, i, j}$ which minimizes $|V(B_{k, i, j})|$.

Suppose that for some $i,j \in [t]$, the block of $(\overrightarrow{B}_{3,i,j}, \g)$ containing $(\dirW_{3,i,j}, \g)$ is $\G_1$-bipartite.
By Lemma~\ref{lemma: one group wall} applied to $(\dirW_{3,i,j}, \g_2)$ in the subgraph $(\overrightarrow{B}_{3, i, j}, \g)$, there exists a flat  $t$-wall $(\dirW_1, \g)$ with certifying separation $(A_1, B_1)$ in $B_{3,i,j}$ such that the block of $(\overrightarrow{B_1}, \g)$ containing $(\dirW_1, \g)$ is $\G_2$-bipartite or $(\dirWone, \g)$ is facially $\G_2$-non-zero. 
As the boundary cycle of $W_{3, i, j}$ is contained in $A_1$, it follows that $(A_1 \cup A_{3, i, j}, B_1)$ is a certifying separation for $W_1$ in $G$.  
Moreover $\zT_{W_1}$ is a restriction of $\zT_{3, i, j}$ and therefore also of $\zT_W$.

Consider the two possible outcomes of Lemma~\ref{lemma: one group wall}.  
In the first case, the block of $(\overrightarrow{B_1}, \g)$ containing $(\dirW_1, \g)$ is $(\G_1 \oplus \G_2)$-bipartite, satisfying $(iii)$.  
In the second case, $(\dirWone, \g)$ is facially $\G_2$-non-zero and the block of $(\overrightarrow{B_1}, \g)$ containing $(\dirWone, \g)$ is $\G_1$-bipartite, satisfying $(ii)$.  
We conclude, by symmetry, 
that we may assume that for all $i,j \in [t]$ and $k \in \{3,4\}$, the subgraph $(\overrightarrow{B}_{k,i,j}, \g)$
is neither $\G_1$-bipartite nor $\G_2$-bipartite.

Suppose next that for some $i,j \in [t]$ and $k \in \{3,4\}$, the block of $(\overrightarrow{B}_{k,i,j}, \g)$ containing $(\dirW_{k,i,j}, \g)$ does not contain a cycle $C$ such that $C$ is 
\begin{itemize}
    \item $\G_1$-non-zero and $\G_2$-zero, or
    \item $\G_1$-zero and $\G_2$-non-zero.
\end{itemize}
By applying Lemma~\ref{lemma: one group wall} to $W_{k, i, j}$ in $(\overrightarrow{B}_{k,i,j}, \g)$ with $\G_1\oplus\G_2$ playing the role of $\G$,
we obtain a $t$-wall $(\dirW_1, \g)$ with certifying separation $(A_1, B_1)$ of $B_{k, i, j}$ such that either the block of $(\overrightarrow{B}_{1}, \g)$ containing  $(\dirWone, \g)$ is $(\G_1 \oplus \G_2)$-bipartite, or
$(\dirWone, \g)$ is facially $(\G_1 \oplus \G_2)$-non-zero.  
As above, in each case, the boundary cycle of $W_{k, i, j}$ is contained in $A_1$ and so $(A_{k, i, j} \cup A_1, B_1)$ is a certifying separation that $W_1$ is flat as a subgraph of $G$.
In the first case, we satisfy conclusion $(iii)$.  
If the latter case holds, we note that $(\dirWone, \g)$ is actually facially $(\G_1, \G_2)$-non-zero, since it does not contain a cycle that is zero in exactly one coordinate, and hence we satisfy conclusion $(i)$.  

Therefore, we may assume that for all $i,j \in [t]$, there is a cycle $C_{3,i,j}$ in the block of $(\overrightarrow{B}_{3, i, j}, \g)$ containing $(\dirW_{3,i,j}, \g)$ that is $\G_k$-zero and $\G_{3-k}$-non-zero for some $k\in [2]$.  
Let $C_{4,i,j}$ be a cycle in the block of $(\overrightarrow{B}_{4,i,j}, \g)$ containing $(\dirW_{4,i,j}, \g)$ that is $\G_k$-non-zero.  

Let $P_{i,j}$ be the subpath of the right vertical path of the brick $D_{i,j}$ with endpoints in distinct horizontal paths of $W_3$.
It is easy to see that there exist four disjoint paths $P_{3,1,i,j},P_{3,2,i,j},P_{4,1,i,j},P_{4,2,i,j}$ such that the following hold.
See Figure~\ref{fig: rerouting} for illustration.
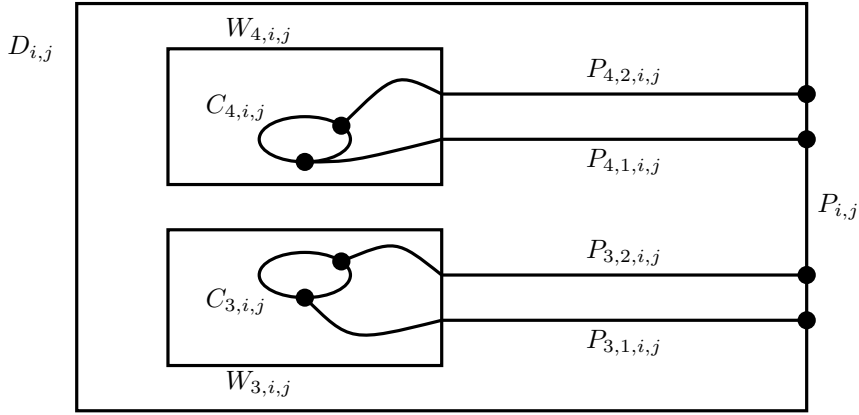
\begin{figure}[t]
\centering
\begin{tikzpicture}
\def\height{1}
\def\width{1}
\def\sizeofwalla{12}
\def\sizeofwallb{12}

\begin{scope}[scale=0.6]

\draw[very thick] (0,0) rectangle (16,9);
\draw[very thick] (2,1) rectangle (8,4);
\draw[very thick] (2,5) rectangle (8,8);

\draw[very thick] (5,3) ellipse (1 and 0.5);
\draw[very thick] (5,6) ellipse (1 and 0.5);

\draw[very thick] (8,2) -- (16,2);
\draw[very thick] (8,3) -- (16,3);
\draw[very thick] (8,6) -- (16,6);
\draw[very thick] (8,7) -- (16,7);

\draw[very thick] (8,2).. controls (6,1.5) .. (5,2.5);
\draw[very thick] (8,3).. controls (7,3.8) .. (5.8,3.3);

\draw[very thick] (8,6).. controls (6,5.5) .. (5,5.5);
\draw[very thick] (8,7).. controls (7,7.5) .. (5.8,6.3);

\draw (-1,8) node {$D_{i,j}$};

\draw (16.7,4.5) node {$P_{i,j}$};

\draw (12,1.5) node {$P_{3,1,i,j}$};
\draw (12,3.5) node {$P_{3,2,i,j}$};
\draw (12,5.5) node {$P_{4,1,i,j}$};
\draw (12,7.5) node {$P_{4,2,i,j}$};

\draw (3.5,2.4) node {$C_{3,i,j}$};
\draw (3.5,6.7) node {$C_{4,i,j}$};

\draw (4,0.6) node {$W_{3,i,j}$};
\draw (4,8.4) node {$W_{4,i,j}$};

\fill[very thick] (16,2) circle (0.2);
\fill[very thick] (16,3) circle (0.2);
\fill[very thick] (16,6) circle (0.2);
\fill[very thick] (16,7) circle (0.2);

\fill[very thick] (5,2.5) circle (0.2);
\fill[very thick] (5.8,3.3) circle (0.2);
\fill[very thick] (5,5.5) circle (0.2);
\fill[very thick] (5.8,6.3) circle (0.2);
\end{scope}

\end{tikzpicture}
\caption{The figures displays the brick $D_{i,j}$ together with the walls $W_{1,i,j}, W_{4,i,j}$ in the proof of Lemma~\ref{lemma: two group wall}.}\label{fig: rerouting}
\end{figure}
\begin{itemize}
	\item $P_{3,1,i,j}$ and $P_{3,2,i,j}$ join $C_{3,i,j}$ and $P_{i,j}$,
	\item $P_{4,1,i,j}$ and $P_{4,2,i,j}$ join $C_{4,i,j}$ and $P_{i,j}$, and
	\item the paths $P_{3,1,i,j},P_{3,2,i,j},P_{4,1,i,j},P_{4,2,i,j}$ end in this order on $P_{i,j}$.
\end{itemize}
Next we reroute similarly as in the proof of Lemma~\ref{lemma: one group wall}.
We start by applying Lemma~\ref{lemma: non-zero brick} 
to the outer cycle of $B_{1,1}$ and $P_{3,1,1,1}$, $P_{3,2,1,1}$, $P_{4,1,1,1}$, $P_{4,2,1,1}$, $C_{3,1,1}$, $C_{4,1,1}$.  
This gives a local rerouting of $P_{1,1}$ such that the brick
$D_{1,1}$ has been modified to a brick that is a $(\G_1,\G_2)$-non-zero.\footnote{
Note that in the application of Lemma~\ref{lemma: non-zero brick}, we may possibly have to interchange the role of $\G_1$ and $\G_2$ in the statement.
}
We then continue iteratively in lexicographic order to find a local rerouting for all bricks $D_{i, j}$ for $i, j \in [t]$.  
Note that once a brick is adjusted, it is never modified again.  At the end, we obtain a flat $(t+1)$-subwall where every facial cycle is non-zero except possibly the last vertical column or horizontal row.  
After deleting the $(t+2)$-horizontal and $(t+2)$-vertical path, we obtain a flat $t$-wall $(\dirW_1, \g)$ of $(\dirG, \g)$ that is facially $(\G_1, \G_2)$-non-zero.  By construction and Proposition \ref{prop: rerouting}, $\zT_{W_1}$ is a restriction of $\zT_W$ as required.  
\end{proof}

\begin{lemma}\label{lemma: zero walls}
Let $t,r\in\N$ such that $r\geq 3t$ and $t \ge 2$.  
Let $\G$ be a group, and let $(\dirG, \g)$ be a $\G$-labeled graph.
Let $W$ be a flat wall in $G$ with certifying separation $(A, B)$ such that $(\overrightarrow{B}, \g)$ is null-labeled
and $W$ $2t$-contains an $r$-subwall $W_1$.
Let $N_1$ be the top nails of $W_1$ with respect to $W$.  Let $(A_1, B_1)$ be a certifying separation for $W_1$.  
Let $X$ be a set of vertices of $G$ meeting all
$\G$-non-zero $N_1$-paths in $(\dirG-(V(B)\setminus N_1), \g)$.
If $|X|\leq 2t-2$, then the $\cT_{W}$-large block of $(\dirG-X, \g)$ is $\G$-bipartite.
\end{lemma}
\begin{proof}
Let $X$ be a set of at most $2t-2$ vertices meeting every $\G$-non-zero $N_1$-path in $(\dirG-(V(B)\setminus N_1), \g)$.
Let $(\dirU, \g)$ be the $\cT_{W}$-large block of $(\dirG-X, \g)$.
For a contradiction, suppose $(\dirU, \g)$ contains a $\G$-non-zero cycle $C$.

Because $|X|\leq 2t-2$ and $W_1$ is $2t$-contained in $W$,
there is a vertical path of $W$ to the left of $W_1$ and one to the right of $W_1$ and also a horizontal path of $W$ below $W_1$ and one above $W_1$, all having an empty intersection with $X$.
Clearly, the union of these four paths contains a cycle $D$ that is contained in $U$.
Moreover, since there are at least two vertical paths of $W$ containing a vertex of $N_1$ and disjoint from $X$, there are two disjoint paths in $W$ linking $D$ and two vertices of $N_1$.
Let the union of these two paths and $D$ be denoted by $\overline{D}$. Since $\overline{D}$ is contained in $(\dirW, \g)$, all edges of $\overline{D}$ are null-labeled. 

Consider the subgraph $B$ of $G$.  By assumption, there exists an $(A\cap B)$-reduction $B'$ of $B$ which embeds in a closed disk $\Delta$.  The graph $B'$ is obtained from $B$ by repeatedly performing elementary $V(A \cap B)$-reductions. Thus the graph $B'$ is a subgraph of $B$ with some additional edges added through the elementary reductions.  
The cycle $D$ corresponds to a cycle $D'$ of $B'$ which bounds a disk $\Delta'$ in $\Delta$.  
We want to consider the subgraph $H$ of $B$ which can be thought of as laying `inside' the disk $\Delta'$.  We begin by defining $H'$ to be the subgraph of $B' \cap B$ embedded in the disc $\Delta'$.  
Define $H$ to be the union of the subgraph $H'$ along with $B^*$ for any separation $(A^*, B^*)$ of $G$ of order at most $3$ where $V(A^* \cap B^*)$ is contained in $V(H')$.  
Observe that the cycle $C$ is not a subgraph of $H$ because $H$ is a subgraph of $B$ and therefore $\G$-bipartite.  Observe as well that $W_1$ is a subgraph of $H$.  

Suppose that $C$ is disjoint from $H$.  
Because $U$ is a block, there are two disjoint paths in $U$ linking $D$ and $C$.
Combining these paths with suitable portions of $\overline{D}$ and $C$ yields a non-zero $N_1$-path in $(\dirG-(V(B_1)\setminus N_1), \g)$ that avoids all vertices in $X$, which is a contradiction.
Therefore, we may assume that $D$ decomposes $C$ into a family of internally disjoint paths $P_1, \dots, P_{2n}$ that alternate between being contained in $H$ and being disjoint from $H$ (except for their endpoints).  
All the paths contained in $H$ are $\G$-zero since $(\overrightarrow{B}, \g)$ is null-labeled.  
Therefore, there is a path $P \in \mathcal{P}$ that is disjoint from $H$ (other than its endpoints) and $\G$-non-zero.
Combining $P$ with a suitable portion of $\overline{D}$, again leads to a $\G$-non-zero $N_1$-path in $(\dirG-(V(B_1)\setminus N_1), \g)$ that avoids all vertices in $X$,
which is a contradiction.
\end{proof}

\section{Proof of our Flat Wall Theorem} \label{sec:thmmain}

We are now ready to prove Theorem \ref{thm: main wall}, which we restate for the reader's convenience.

\begin{reptheorem}{thm: main wall}
For every $t\in \N$, there exist integers $T(t)$ and $g(t)$
with the following property. 
Let $\G_1$ and $\G_2$ be  groups and $(\dirG, \g)$ be a $(\G_1 \oplus \G_2)$-labeled graph.
If $(\dirG, \g)$ contains a $T(t)$-wall $(\dirW, \g)$,
then one of the following statements holds.
\begin{enumerate}[(a)]
	\item There is a $(\G_1,\G_2)$-odd $K_t$-model $\pi$ in $(\dirG, \g)$ such that $\cT_\pi$ is a restriction of $\cT_W$.
	\item There is a set of vertices $Z$
	such that $|Z|\leq g(t)$ and there is a flat $100t$-wall $(\overrightarrow{W_0}, \g)$ in $(\dirG-Z, \g)$ with top nails $N_0$ and certifying separation	$(A_0, B_0)$ such that
	$\cT_{W_0}$ is a restriction of $\cT_W$ and after possibly shifting
	\begin{enumerate}
		\item[(b.i)] $(\overrightarrow{W_0}, \g)$ is facially $(\G_1,\G_2)$-non-zero, or
		\item[(b.ii)] for some $i\in [2]$,  $(\overrightarrow{B_0}, \g_i)$ is null-labeled, $(\overrightarrow{W_0}, \g)$ is facially $\G_{3-i}$-non-zero,
		and there is a $\G_i$-clean $N_0$-linkage $\cP$ of size $t$ 
		with respect to $(A_0,B_0)$, or
		\item[(b.iii)] $(\overrightarrow{B_0}, \g)$ is null-labeled
		and there is a $(\G_1,\G_2)$-clean pair $(\cP,\cQ)$ of $N_0$-linkages of size $t$ with respect to $(A_0,B_0)$.
  \end{enumerate}
  \item There is a set $Z$ of vertices of $(\dirG, \g)$ and some $i\in [2]$ such that $|Z|\leq g(t)$ and the $\cT_W$-large block of $$(\dirG-Z, \g)$$ is $\G_i$-bipartite.
\end{enumerate}
\end{reptheorem}

\begin{proof}
Let $\cT_W$ be the tangle in $G$ induced by $W$.  
We first suppose that there is a $T(t)$-clique minor $K$ in $G$ (where $T(t)$ is the function from Lemma \ref{lemma: odd Kt minor2}) such that $\cT_K$ is a restriction of $\cT_W$.  
By Lemma \ref{lemma: odd Kt minor2}, there is a $(\G_1,\G_2)$-odd $K_t$-model $\pi$ in $(\dirG, \g)$ such that $\cT_\pi$ is a restriction of $\cT_W$ or 
there is a set $Z$ of vertices (with $|Z|$ bounded by a function of $t$) and an index $i \in [2]$ such that the $\cT_K$-large block of $(\dirG-Z, \g)$ is $\G_i$-bipartite.  
Since $\cT_K$ is a restriction of $\cT_W$, this block is also $\cT_W$-large, so we are done.  

Therefore, we may assume that $\cT_W$ does not control a large clique-minor.  
By the Flat Wall Theorem (Theorem~\ref{flatwalltheorem}), 
there is a set of
vertices $Z_1$ such that $|Z_1|$ is bounded in terms of $t$ and $W$ contains a $10^{22}t^{18}$-subwall $W_1$ which is flat in $G_1=G - Z_1$.  

By Lemma \ref{lemma: two group wall}, there exists a flat $10^7t^6$-wall $W_2$ with certifying separation $(A_2, B_2)$ such that 
\begin{enumerate}[(i)]
	\item $(\dirW_2, \g)$ is facially $(\G_1,\G_2)$-non-zero, or
	\item $(\dirW_2, \g)$ is facially $\G_i$-non-zero and $(\overrightarrow{B_2}, \g)$ is $\G_{3-i}$-bipartite for some $i\in [2]$, or
	\item $(\overrightarrow{B_2}, \g)$ is $(\G_1 \oplus \G_2)$-bipartite.
\end{enumerate}  
Moreover, $\zT_{W_2}$ is a restriction of $\zT_{W_1}$.  

If $(\dirW_2, \g)$ is facially $(\G_1,\G_2)$-non-zero, then we are done.  
Otherwise, let $W_3$ be a $10^6t^6$-subwall of $W_2$ that is $1$-contained in $W_2$.
Let $N_3$ be the top nails of $W_3$ with respect to $W_2$.  
The wall $W_3$ is flat since it is a subwall of $W_2$. Fix a certifying separation $(A_3, B_3)$ for $W_3$ with $|V(B_3)|$ minimum.  

Let $H$ be a component of $A_2$ containing a vertex of $A_2 \cap B_2$.  The union of $H$ along with the component of $W_2 - V(W_3)$ containing the boundary cycle of $W_2$ is a connected subgraph of $G$ which is disjoint from $W_3$ and has as neighbors every nail and corner of $W_3$ in the boundary cycle of $W_3$.  It follows from Proposition~\ref{lemma: boundary cycle excluded} that $H$ is a subgraph of $A_3$.  
We conclude that $B_3$ is a subgraph of $B_2$.  

Suppose (ii) holds; that is, $(\dirW_2, \g)$ is facially $\G_i$-non-zero and $(\overrightarrow{B_2}, \g)$ is $\G_{3-i}$-bipartite for some $i\in [2]$.  
We perform shifts so that all edges in $(\overrightarrow{B_2}, \g_{3-i})$ are null-labeled.
Let $W_4$ be a $10t$-contained $100t$-subwall of $W_3$
(such that every brick of $W_4$ is a brick of $W_3$)
 with certifying separation $(A_4,B_4)$
such that $|V(B_4)|$ is minimal.
Let $N_4$ be the top nails of $W_4$ with respect to $W_3$.  

If there is a set of vertices $Z_2$ of size at most $2\cdot 10^5t^6$ meeting all $\G_{3-i}$-non-zero $N_3$-paths in $(\dirG_1-(V(B_3) \setminus N_3), \g)$,
then by Lemma~\ref{lemma: zero walls}, the $\cT_{W_3}$-large block of $(\dirG_1-Z_2, \g)$  is $\G_{3-i}$-bipartite. Therefore, $Z=Z_1\cup Z_2$ satisfies the third outcome of the theorem.

Thus, we may suppose that such a set $Z_2$ does not exist.
Applying Lemma~\ref{lemma: non-zero A-paths} gives $10^5t^6$ disjoint $\G_{3-i}$-non-zero $N_3$-paths in $(\dirG_1-(V(B_3) \setminus N_3), \g)$.  
Applying Lemma~\ref{lemma: wall paths1} yields a $\G_{3-i}$-clean $W_4$-linkage of size $t$, as required.

Therefore, we may assume that (iii) holds; that is, $(\overrightarrow{B_2}, \g)$ is $(\G_1 \oplus \G_2)$-bipartite.  
We perform shifts so that all edges in $(\overrightarrow{B_2}, \g)$ are null-labeled.  As $B_3$ is a subgraph of $B_2$, all the edges in $(\overrightarrow{B_3}, \g)$ are null-labeled.  By the previous argument, we may assume that there are $10^5t^6$ disjoint $\G_{i}$-non-zero $N_3$-paths in $(\dirG_1-(V(B_3) \setminus N_3), \g)$ for both $i\in[2]$.
Applying Lemma~\ref{lemma: wall paths} yields a $(\G_1,\G_2)$-clean pair of $N_4$-linkages of size $t$, which completes the proof of the theorem.
\end{proof}

\section{Deriving the Erd\H{o}s-P\'osa results} \label{sec:deriving}

We finish our paper by deriving Theorem~\ref{thm: EP hi cycles} and \ref{thm: EP groups sparse} 
from Theorem \ref{thm: main wall}. We also give some additional applications. For the reader's convenience we restate both theorems. 
\begin{reptheorem}{thm: EP hi cycles}
For every integer $k$, there exists an integer $f(k)$ with the following property.  
Let $\G_1$ and $\G_2$ be  groups
and let $(\dirG, \g)$ be a $(\G_1 \oplus \G_2)$-labeled graph.  
Then, $(\dirG, \g)$ contains $k$ $(\G_1,\G_2)$-non-zero cycles such that each vertex of $(\dirG, \g)$ is in at most two of these cycles, 
or there exists a set of at most $f(k)$ vertices of $G$ such that $(\dirG-X, \g)$ does not contain any $(\G_1,\G_2)$-non-zero cycle.  
\end{reptheorem}

\begin{proof}
Let $f$ be a fast growing function to be specified later.
Assume for a contradiction that $((\dirG, \g),k)$ is a minimal counterexample showing that $f$ is not a half-integral \EP{} function for the set of $(\G_1 \oplus \G_2)$-labeled graphs.  By Lemma~\ref{lemma: EP tangle} and Theorem~\ref{lemma: tangle wall}, there is a tangle $\cT$ of order $T_a$ and a $T_b$-wall $W_1$ in $G$ such that $\cT_{W_1}$ is a restriction of $\cT$.  
We specify $T_b$ later and choose $T_a$ sufficiently large in terms of $T_b$, as required by Theorem~\ref{lemma: tangle wall}.

Next we apply Theorem~\ref{thm: main wall} to $(\dirW_1, \g)$ and choose $T_b$ sufficiently large so that we can apply Theorem~\ref{thm: main wall} with $t=6k$.
If there is a $(\G_1,\G_2)$-odd $K_{6k}$-minor in $(\dirG, \g)$, 
then, by Lemma~\ref{lemma: two odd triangles} (applied $k$ times to two disjoint cycles given by two disjoint triangles of the $K_{6k}$-model), 
there are $k$ disjoint $(\G_1,\G_2)$-non-zero cycles in $(\dirG, \g)$, which is a contradiction to our assumption.

If there is a set $Z \subseteq V(G)$ such that  $|Z|\leq g(6k)$ (the function $g$ of Theorem~\ref{thm: main wall}) and the $\cT$-large block of $(\dirG-Z, \g)$ is $\G_i$-bipartite for some $i\in[2]$,
then $(\dirG-Z, \g)$ does not contain any $(\G_1,\G_2)$-non-zero cycle, which is a contradiction to our assumption.

Therefore, by Theorem~\ref{thm: main wall}, we may assume that there exists a set $Z \subseteq V(G)$ such that $|Z|\leq g(6k)$ and $G-Z$ contains a flat $600k$-wall $W_2$ with certifying separation $(A_2, B_2)$ and top nails $N_2$ such that
$\cT_{W_{2}}$ is a restriction of $\cT_{W_1}$ and after possibly shifting
	\begin{itemize}
		\item $(\overrightarrow{W_{2}}, \g)$ is facially $(\G_1,\G_2)$-non-zero, or
		\item for some $i\in [2]$,  the wall $(\overrightarrow{B_2}, \g_{i})$ is null-labeled, $(\overrightarrow{W_2}, \g)$ is facially $\G_{3-i}$-non-zero,
		and there is a $\G_i$-clean $N_2$-linkage $\cP$ with respect to $(A_2,B_2)$ of size $k$, or
		\item $(\overrightarrow{B_2}, \g)$ is null-labeled
		and there is a $(\G_1,\G_2)$-clean pair $(\cP,\cQ)$ of $N_2$-linkages of size $2k$
		with respect to $(A_2,B_2)$.
  \end{itemize}

If  $(\overrightarrow{W_{2}}, \g)$ is facially $(\G_1,\G_2)$-non-zero, then clearly $(\dirG, \g)$
contains $k$ disjoint $(\G_1,\G_2)$-non-zero cycles, which is a contradiction.

Suppose next, by symmetry, that $(\overrightarrow{B_2}, \g_1)$ is null-labeled, $(\overrightarrow{W_{2}}, \g)$ is facially $\G_{2}$-non-zero,
and there is a $\G_1$-clean $N_2$-linkage $\cP=\{P_1, \dots, P_k\}$.
The endpoints of the paths in $\cP$ inherit a linear ordering from $N_2$.  We extend each $P_i$ to a path $P_i'$ as follows.  
First extend each $P_i$ from each of its ends to the next smaller branch vertex (or corner) of $W_2$.
Next, extend each of these new paths via the vertical paths of $W_2$ so that $P_i'$ crosses exactly $10i$ horizontal paths and for each horizontal path $H$ which intersects $P_i'$, the graph $H\cap P_i'$ has two components.
Let $a(i)$ be the integer such that the endpoints of $P_i'$ lie on the horizontal path $P_{a(i)}(h)$ of $W_2$.
Let $H_i$ be the subpath of $P_{a(i)}(h)$ that joins the ends of $P_i'$, and let $D_i$ be a brick of $W_2$ intersecting $H_i$ in at least one edge.  Since $D_i$ is a $\G_2$-non-zero cycle, either $P_i' \cup H_i$ or $(P_i' \cup H_i) \Delta D_i$ is a $(\G_1, \G_2)$-non-zero cycle where $\Delta$ denotes the symmetric difference.  In total, we have constructed $k$ $(\G_1,\G_2)$-non-zero cycles, and every vertex of $G$ is contained in at most two cycles, which is a contradiction. 

Finally, suppose $(\overrightarrow{B_2}, \g)$ is null-labeled
		and there is a $(\G_1,\G_2)$-clean pair $(\cP,\cQ)$ of $N_2$-linkages of size $2k$.
If $\cP$ or $\cQ$ contain $k$ paths that are $(\G_1, \G_2$)-non-zero,
then using these $k$ paths and the null-labeled wall ${W_2}$ in the same construction as the previous paragraph yields a set of $k$ $(\G_1,\G_2)$-non-zero cycles such that every vertex is in at most two of the cycles, a contradiction.
Hence, we may assume that there is a $(\G_1,\G_2)$-clean pair of $N_2$-linkages $(\cP',\cQ')$ of size $k$ such that every path in $\cP'$ is $\G_1$-non-zero and $\G_2$-zero and every path in $\cQ'$ is $\G_1$-zero and $\G_2$-non-zero.  
Extending the paths in $\cP'\cup \cQ'$  similarly as in the previous case yields $k$  $(\G_1,\G_2)$-non-zero cycles each containing exactly one path of $\cP$, one path of $\cQ$, and the subpaths of two distinct horizontal paths of $W_2$. 
The cycles can be constructed such that every vertex of $G$ is contained in at most two of these cycles, which is the final contradiction.
\end{proof}

\begin{reptheorem}{thm: EP groups sparse}
For every integer $k$, there exists an integer $f(k)$ with the following property.  
Let $\G_1$ and $\G_2$ be  groups 
and let $(\dirG, \g)$ be a \sparse\ $(\G_1 \oplus \G_2)$-labeled graph.
Then, $(\dirG, \g)$ contains $k$ disjoint $(\G_1,\G_2)$-non-zero cycles or there exists a set of at most $f(k)$ vertices of $(\dirG, \g)$
such that $(\dirG-X, \g)$ does not contain any $(\G_1,\G_2)$-non-zero cycle.  
\end{reptheorem}

\begin{proof}
The proof is similar to the proof of Theorem \ref{thm: EP hi cycles}.  
Let $f$ be a fast growing function to be specified later.
Assume for a contradiction that $((\dirG, \g),k)$ is a minimal counterexample showing that $f$ is not an \EP{} function for the set of \sparse{} $(\G_1 \oplus \G_2)$-labeled graphs.  By Lemma~\ref{lemma: EP tangle} and Theorem~\ref{lemma: tangle wall}, there is a tangle $\cT$ of order $T_a$ and a $T_b$-wall $W_1$ in $G$ such that $\cT_{W_1}$ is a restriction of $\cT$.  
We specify $T_b$ later and choose $T_a$ sufficiently large in terms of $T_b$, as required by Theorem~\ref{lemma: tangle wall}.

Next we apply Theorem~\ref{thm: main wall} to $(\dirW_1, \g)$ and choose $T_b$ sufficiently large so that we can apply Theorem~\ref{thm: main wall} with $t=6k$.
If there is a $(\G_1,\G_2)$-odd $K_{6k}$-minor in $(\dirG, \g)$, then by Lemma~\ref{lemma: two odd triangles}, there are $k$ disjoint $(\G_1,\G_2)$-non-zero cycles in $(\dirG, \g)$, which is a contradiction to our assumption.

If there is a set $Z \subseteq V(G)$ such that  $|Z|\leq g(6k)$ (the function $g$ of Theorem~\ref{thm: main wall}) and the $\cT$-large block of $(\dirG-Z, \g)$ is $\G_i$-bipartite for some $i\in[2]$,
then $(\dirG-Z, \g)$ does not contain any $(\G_1,\G_2)$-non-zero cycle, which is a contradiction to our assumption.

Therefore, by Theorem~\ref{thm: main wall}, we may assume that
there exists a set $Z \subseteq V(G)$ such that $|Z|\leq g(6k)$ and $G-Z$ contains a flat $600k$-wall $W_2$ with certifying separation $(A_2, B_2)$ and top nails $N_2$ such that
$\cT_{W_{2}}$ is a restriction of $\cT_{W_1}$ and after possibly shifting
\begin{itemize}
    \item $(\overrightarrow{W_{2}}, \g)$ is facially $(\G_1,\G_2)$-non-zero, or
	\item for some $i\in [2]$,  $(\overrightarrow{B_2}, \g_i)$ is null-labeled, $(\overrightarrow{W_{2}}, \g)$ is facially $\G_{3-i}$-non-zero,
	and there is a $\G_i$-clean $N_2$-linkage $\cP$ of size $k$ with respect to $(A_2,B_2)$, or
	\item $(\overrightarrow{B_2}, \g)$ is null-labeled
	and there is a $(\G_1,\G_2)$-clean pair  $(\cP,\cQ)$ of $N_2$-linkages with respect to $(A_2,B_2)$ of size $2k$.
\end{itemize}

If  $(\overrightarrow{W_{2}}, \g)$ is facially $(\G_1,\G_2)$-non-zero, then clearly $(\dirG, \g)$
contains $k$ disjoint $(\G_1,\G_2)$-non-zero cycles, which is a contradiction.

Suppose next, by symmetry, that $(\overrightarrow{B_2}, \g_1)$ is null-labeled, $(\overrightarrow{W_{2}}, \g)$ is facially $\G_{2}$-non-zero,
		and there is a $\G_1$-clean $N_2$-linkage $\cP$ of size $k$. If $\cP$ is in series or nested, then it is easy to see that we can extend $k$ paths in $\cP$ to a set of $k$ disjoint $(\G_1,\G_2)$-non-zero cycles, which is a contradiction.  Observe that $\cP$ cannot be crossing as $(\dirG, \g)$ is \sparse{}.
To see this,
observe that if $\cP$ is crossing, then for all paths $P_1,P_2\in \cP$, we have $\g_1(P_1)=\g_1(P_2)$.
Therefore, it is easy to construct two cycles $C_1,C_2$ such that $C_i$ consists of $P_i$ and a (null-labeled) path in $W_2$ for each $i\in[2]$.
These two cycles contradict the fact $(\dirG,\g)$ is \sparse{}.

Finally, suppose $(\overrightarrow{B_2}, \g)$ is null-labeled and there is a $(\G_1,\G_2)$-clean pair of $N_2$-linkages $(\cP,\cQ)$ of size $2k$. 
If $\cP$ (or $\cQ$) contains $k$ paths that are $(\G_1, \G_2)$-non-zero, then we can extend these paths to $k$ disjoint $(\G_1, \G_2)$-non-zero cycles.  Note that in order to do so, we need that $\cP$ is not crossing, which holds given that $(\dirG, \g)$ is \sparse{}.  
  
Hence we may assume that there is a $(\G_1,\G_2)$-clean pair of $N_2$-linkages $(\cP',\cQ')$ of size $k$ such that each path in $\cP'$ is $\G_1$-non-zero and $\G_2$-zero, and each path in $\cQ'$ is $\G_1$-zero and $\G_2$-non-zero.  
If $\cP'$ and $\cQ'$ are of the same type, then there exist $k$ disjoint $(\G_1, \G_2)$-non-zero cycles, where each cycle uses exactly one path from each of $\cP'$ and $\cQ'$.  We conclude that, by symmetry, $\cP'$ is in series and $\cQ'$ is nested.  
However, this is a contradiction, because $\cQ'$ cannot be nested as $(\dirG, \g)$ is \sparse{}.
\end{proof}	

\textbf{Remark.} In the proof of Theorem \ref{thm: EP groups sparse}, the robustness of $(G, \g)$  is only used in three places.  
The first is if the clean $N_2$-linkage $\cP$ is a crossing linkage.  
The second is if in the $(\G_1,\G_2)$-clean pair of $N_2$-linkages $(\cP, \cQ)$, one of $\cP$ or $\cQ$ is crossing and contains $k$ paths that are $(\G_1, \G_2)$-non-zero.  
The third is if the $(\G_1,\G_2)$-clean pair of $W$-linkages $(\cP',\cQ')$ are of different types. 
If we allow these as separate outcomes, then we obtain the following theorem for the full \EP{}-property for $(\G_1,\G_2)$-non-zero cycles. 

\begin{theorem} \label{doubleEscher}
For every integer $k$, there exists an integer $f(k)$ with the following property.  
For all  groups $\G_1$ and $\G_2$, and all $(\G_1 \oplus \G_2)$-labeled graphs $(\dirG, \g)$, at least one of the following holds.
\begin{enumerate}[(i)]
    \item 
    $(\dirG, \g)$ contains $k$ disjoint $(\G_1, \G_2)$-non-zero cycles.
    \item
    There exists a set of at most $f(k)$ vertices of $G$ such that $(\dirG-X, \g)$ does not contain any $(\G_1,\G_2)$-non-zero cycle.
\end{enumerate}
$(\dirG, \g)$ contains a $600k$-wall $(\dirW, \g)$ with top nails $N$ satisfying one of the following.
\begin{enumerate}
    \item [(iii)]
    There exists a crossing $N$-linkage $\cP$ of size $k$ which is internally disjoint from $W$ such that after possibly shifting there exists an index $i \in [2]$ such that $(\dirW, \g_i)$ is null-labeled, $(\dirW, \g)$ is facially $\G_{3-i}$-non-zero, and for all $P_1,P_2\in \cP$, we have $\g_i(P_1)=\g_i(P_2)\neq 0$, or
    \item[(iv)]
    there exists a crossing $N$-linkage $\cP$ of size $k$ which is internally disjoint from $W$ such that after possibly shifting $(\dirW, \g)$ is null-labeled, and each path in $\cP$ is $(\G_1, \G_2)$-non-zero, or
    \item[(v)]
    there exists a pair of $N$-linkages $(\cP, \cQ)$ of size $k$ such that $\cP$ and $\cQ$ are of different type and after possibly shifting 
    \begin{itemize}
        \item $(\cP,\cQ)$ is a $(\G_1,\G_2)$-clean $N$-linkage,
        \item 
        $(\dirW, \g)$ is null-labeled, and
        \item
        $P$ is $\G_1$-zero and $\G_2$-non-zero for all $P\in\cP$, and
        \item
         $Q$ is $\G_1$-non-zero and $\G_2$-zero for all $Q\in\cQ$.
    \end{itemize}
\end{enumerate}
\end{theorem}

Let $G$ be a graph and $S \subseteq V(G)$.  Let $\overrightarrow{G}$ be an arbitrary orientation of $G$ and let $e_1, \dots, e_m$ be an enumeration of $E(\dirG)$.  Define $\g: E(\overrightarrow{G}) \to \Z / 2\Z \oplus \Z$ by $\g(e_i)=(1, 2^i)$, if $e_i$ has at least one end in $S$, and $\g(e_i)=(1,0)$, otherwise. By applying Theorem \ref{doubleEscher} to $(\dirG, \g)$, we obtain the following theorem for odd $S$-cycles.

\begin{theorem} \label{oddSobstructions}
For every positive integer $k$, there exist an integer $f(k)$ with the following property.  
For all graphs $G$ and all $S \subseteq V(G)$ at least one of the following holds.
\begin{enumerate}[(i)]
	\item $G$ contains $k$ disjoint odd $S$-cycles,
	\item there exists a set $X \subseteq V(G)$ such that $G-X$ does not contain an odd $S$-cycle and $|X|\leq f(k)$, or
	\item $G$ contains a bipartite $600k$-wall $W$ with top nails $N$ such that every face of $W$ contains a vertex of $S$, 
	and there is a crossing $N$-linkage $\mathcal{P}$ of size $k$ which is internally disjoint from $W$ such that $W\cup P$ is not bipartite for all $P\in \cP$, or
	\item $G$ contains a bipartite $600k$-wall $W$ with top nails $N$ such that $V(W) \cap S=\emptyset$,  
	and there is a crossing $N$-linkage $\mathcal{P}$ of size $k$ which is internally disjoint from $W$ such that $W\cup P$ is not bipartite and $V(P)\cap S\neq \emptyset$ for all $P\in \cP$,	or
	\item
	$G$ contains a bipartite $600k$-wall $W$ with top nails $N$ such that 
	$V(W) \cap S=\emptyset$, and there is a pair of pure $N$-linkages $\mathcal{P}$ and $\mathcal{Q}$ each of size $k$ such that 
    \begin{itemize}
    \item $\cP\cup \cQ$ is a linkage of size $2k$,
    \item $\cP$ is in series and $\cQ$ is either crossing or nested,
    \item $V(P) \cap S \neq \emptyset$ and $P\cup W$ is bipartite for all $P\in\cP$,
    \item $V(Q) \cap S=\emptyset$ and $Q\cup W$ is not bipartite for all $Q\in\cQ$, and
	\item $I_{\cP}\cap I_{\cQ}=\emptyset$. 
\end{itemize}
\end{enumerate}
\end{theorem}

\textbf{Remark.} If $S=V(G)$, the last two outcomes of Theorem \ref{oddSobstructions} obviously cannot occur.  Therefore, we obtain an independent proof of Reed's Escher-wall Theorem (Theorem \ref{escherwalls}).

As a last application, we present the canonical set of obstructions for cycles not homologous to zero. 
Let $\Sigma$ be a surface, $\mathcal{H}(\Sigma)$ be the (first) homology group of $\Sigma$, and $G$ be a graph embedded in $\Sigma$. Recall the construction of a $(\mathcal{H}(\Sigma) \oplus \mathcal{H}(\Sigma))$-labeled graph $(\dirG, \g)$ from $G$ as follows. Let $\dirG$ be an arbitrary orientation of $G$.  Fix a spanning forest $F$ of $G$, and define $\g(e)=(0,0)$ for all $e \in \overrightarrow{F}$. For $e \notin \overrightarrow{F}$, define $\g(e)$ to be $(\alpha, \alpha)$, where $\alpha$ is the homology class of the unique cycle contained in $e \cup F$. If we apply Theorem \ref{doubleEscher} to $(G, \g)$, then outcome $(iii)$ and $(v)$ cannot occur since $\g_1(C)=\g_2(C)$ for all cycles $C$ in $(G, \g)$. Therefore, we immediately obtain the following theorem. 

\begin{theorem} \label{homology obstructions}
For every integer $k$ there exists an integer $f(k)$ with the following property.  For every surface $\Sigma$ and every graph $G$ embedded in $\Sigma$ at least one of the following holds.
\begin{enumerate}[(i)]
    \item $G$ contains $k$ disjoint cycles not homologous to zero, 
    \item there is a set of at most $f(k)$ vertices of $G$ such that all cycles of $G-X$ are homologous to zero, or
    \item $G$ contains a $600k$-wall $W$ with top nails $N$ and a crossing $N$-linkage $\mathcal{P}$ of size $k$ which is internally disjoint from $W$ such that each path in $\mathcal{P}$ goes through the same crosscap of $\Sigma$.  
\end{enumerate}
\end{theorem}

As an immediate corollary of Theorem \ref{homology obstructions}, we obtain the (full) \EP{} property for cycles not homologous to zero on an orientable surface.  

\begin{corollary} \label{orientable homology}
For every integer $k$ there exists an integer $f(k)$ with the following property.  For every orientable surface $\Sigma$ and every graph $G$ embedded on $\Sigma$, either $G$ contains $k$ disjoint cycles not homologous to zero, or there is a set of at most $f(k)$ vertices of $G$ such that all cycles of $G-X$ are homologous to zero.
\end{corollary}

\textbf{Remark.} Fix $\ell \in \mathbb{N}$.  We call a cycle in a $(\G_1,\G_2)$-group-labeled graph $(\dirG, \g)$ \emph{long} if it has length at least $\ell$.
Note that the proofs of Theorem~\ref{thm: EP hi cycles} and \ref{thm: EP groups sparse} can easily be adapted for $(\G_1,\G_2)$-non-zero cycles which are also long,
by applying Theorem~\ref{thm: main wall} with $t=2\ell k$ instead of $t=6k$.

\section*{Acknowledgements} We would like Robin Thomas and Youngho Yoo for spotting an error in the proof of Lemma~\ref{lemma: odd Kt minor} in the published version of this article.  We have corrected the error in this version.


 

\vfill

\small
\vskip2mm plus 1fill
\noindent
Version \today{}
\bigbreak

\noindent
Tony Huynh
{\tt <tony.bourbaki@gmail.com>}\\
Département de Mathématique\\
Université Libre de Bruxelles\\
Belgium\\

\noindent
Felix Joos\\
Institut f\"ur Optimierung und Operations Research\\
Universit\"at Ulm\\
Germany\\

\noindent
Paul Wollan
{\tt <wollan@di.uniroma1.it>}\\
Department of Computer Science\\
University of Rome\\
Italy\\

\end{document}